\documentclass[12pt,reqno]{amsart}

\usepackage{hyperref}
\usepackage{color,enumerate}
\numberwithin{equation}{section} \numberwithin{figure}{section}
\numberwithin{table}{section} 

{ \theoremstyle{plain}
\newtheorem{theorem}{Theorem}[section]
\newtheorem{proposition}[theorem]{Proposition}
\newtheorem{lemma}[theorem]{Lemma}
\newtheorem{corollary}[theorem]{Corollary}
\newtheorem{conjecture}[theorem]{Conjecture}
\newtheorem{remark}[theorem]{Remark}
\newtheorem{definition}[theorem]{Definition}
}

\definecolor{rosso}{rgb}{0.7,0,0}
\definecolor{blu}{rgb}{0,0,1}

\def\R{\mathbb{R}}

\def\N{\mathbb{N}}

\def\a{\alpha}
\def\ep{\varepsilon}
\def\e{\varepsilon}
\def\b{\beta}
\def\cone{\mathcal{C}}
\def\z{\zeta}

\begin{document}

\title[Multi-layer radial solutions]{Multi-layer radial solutions for a supercritical Neumann problem}
\thanks{D.B. is supported by INRIA - Team MEPHYSTO, MIS F.4508.14 (FNRS), PDR T.1110.14F (FNRS) 
\& ARC AUWB-2012-12/17-ULB1- IAPAS. M. Grossi and S. Terracini are partially supported by PRIN-2012-grant ``Variational and perturbative aspects of nonlinear differential problems''. S. Terracini is partially supported by INDAM. B. Noris and S. Terracini are partially supported by the project ERC Advanced Grant  2013 n. 339958: ``Complex Patterns for Strongly Interacting Dynamical Systems - COMPAT''}
\author[Bonheure]{Denis Bonheure}
\address{D\'epartement de mathmatique, Universit\'e libre de Bruxelles, Bruxelles, Belgium}
\email{denis.bonheure@ulb.ac.be }
\author[Grossi]{Massimo Grossi}
\address{Dipartimento di Matematica, Universit\`a di Roma
La Sapienza, P.le A. Moro 2 - 00185 Roma- Italy.}
\email{massimo.grossi@uniroma1.it}
\author[Noris]{Benedetta Noris}
\address{D\'epartement de mathmatique, Universit\'e libre de Bruxelles, Bruxelles, Belgium}
\email{benedettanoris@gmail.com}
\author[Terracini]{Susanna Terracini}
\address{Dipartimento di Matematica``Giuseppe Peano", Universit\`a di Torino, Via Carlo Alberto 10, 10123. Torino, Italy.}
\email{susanna.terracini@unito.it.}
\date{\today}

\begin{abstract}
In this paper we study the Neumann problem
\begin{equation*}
\begin{cases}
-\Delta u+u=u^p & \text{ in }B_1 \\
u > 0, & \text{ in }B_1 \\
\partial_\nu u=0 & \text{ on } \partial B_1,
\end{cases}
\end{equation*}
and we show the existence of multiple-layer radial solutions as $p\rightarrow+\infty$.
\end{abstract}

\maketitle

\section{Introduction}
\subsection{Motivations and main results}
This paper is concerned with multiple solutions of the following problem, 
\begin{equation}\label{main}
\left\{\begin{array}{ll}
-\Delta u+u=u^p & \text{ in } B_1\\
u>0\\
\partial_\nu u=0 & \text{ on } \partial B_1,
\end{array}
\right.
\end{equation}
where $B_1$ is the unitary ball in $\R^N$, $N\geq3$, and $p>1$.\par
Such simple models, coming from a variety of applications, have started and inspired the analysis of singular behavior in nonlinear elliptic partial differential equations in the last two decades (see, e.g. \cite{DF,Ni1998}).  The typical situation is when, for  limiting values of a certain parameter,  there are special  solutions exhibiting a varied limiting behavior.  Here we are concerned with the asymptotic  $p\to+\infty$. In this, as well as in many other problems, one of the main points of the the analysis is the identification of its singular limits. Here we shall follow this strategy, in our search for solutions showing multiple oscillations for problem \eqref{main}.

\par

This particular problem has attracted much interest in recent years because, in spite of its simple and apparently harmless form, it already shows a variety of interesting phenomena. Just to start with, the very same existence of solution is  extremely sensitive to the boundary conditions: indeed, as well known, by the Poho\v{z}aev identity \cite{P}, the Dirichlet problem has no solution for $p\ge\frac{N+2}{N-2}$. On the other hand, the situation changes drastically when dealing with Neumann boundary conditions, when, even in the {\em supercritical regime} $p\ge\frac{N+2}{N-2}$  there hold existence results (\cite{ST,BS,BonheureSerra2013, BNW} )

\par
 Let us start recalling that in \cite{ST} it has been showed that the problem
\begin{equation}\label{I1}
\left\{\begin{array}{ll}
-\Delta u+u=a(|x|)u^p & \text{ in } B_1\\
u>0\\
\partial_\nu u=0 & \text{ on } \partial B_1,
\end{array}
\right.
\end{equation}
where $a\in L^1(B_1)$ is increasing, not constant and $a(r) >0$ a.e. in $[0,1]$ {\it admits at least one radially increasing solution}. It is a very remarkable fact that this holds irrespective of the sub or supercritical character of the power $p$. This result was extended in \cite{BNW} to the case of $a\equiv1$.\par
Other progresses have been made when the power $p$ tends to $+\infty$. In \cite{GN} it was shown the existence of a radial solution $u_p$ to \eqref{main} which satisfies
\begin{equation}\label{I2}
u_p(|x|)\rightarrow\frac{G(|x|,1)}{G(1,1)}
\end{equation}
where $G(r,s)$ is the Green function associated to the one dimensional operator
\begin{equation}\label{eq:L_mathcal}
{\mathcal L}: u\mapsto-u''-\frac{N-1}ru'+u,
\end{equation}
for the boundary conditions $u'(0)=u'(1)=0$ (see also \cite{GladialiGrossi2007}). Note that \eqref{I2} can be read as a {\em concentration on }$S^{N-1}$. Indeed, it can be shown  that in this case we have that the terms $u^p$ weakly converge to a multiple of the ($N-1$)-dimensional Hausdorff measure supported on the $1$-sphere.\par
In the case of the {\em annulus}, in \cite{BS} it was shown the existence of at least three different nonradial solutions to \eqref{main} as $p$ goes to $+\infty$.
These are {\it single or double layer solutions}, as their laplacian blows up in one --or at most two-- annuli about certain spheres, while in the rest of the domain there holds full $\mathcal C^2$ convergence.
\par

The aim of this paper is to prove the existence {\it multiple layer solutions}, that is radial solutions to  \eqref{main} whose laplacians weakly converge to measures concentrating at interior spheres, with a simple reflection rule. The existence of multiple layer solutions was found, for different singularly perturbed problems and various boundary conditions, in recent papers (see for example \cite{AMN1, AMN2,MalchiodiNiWei2005,MalchiodiMontenegro2002, BandleWei2008, WeiYan2006,RS}).\par

We shall exploit a gluing technique, using a variant of Nehari method, adapted to deal with Neumann problems instead of the standard Dirichlet ones: we choose a partition of $(0,1)$ given by $0<t_1<..<t_k<1$ and consider in $(0,t_1)$ the {\em increasing} solution obtained in \cite{BNW} and in $(t_{i-1},t_i)$ the solutions found in \cite{BS}. \par
Of course, this gluing procedure provides a solution in $(0,1)$ if and only if the value of the solutions at the endpoints $t_i$ coincide. This will be true for a careful choice of the partition, related with an auxiliary variational problem. \par
Note that our approach is very different from others dealing with existence of multiple layers radial solutions. In our opinion it is simpler and it could be applied to various perturbative problems. 
As a counterpart, it needs some careful expansions of the solutions in \cite{BNW} and in \cite{BS} as $p$ goes to $+\infty$. We finally recall that solutions featuring highly oscillatory behaviour have been studied, among others, in \cite{TV1,TV2,OrtegaVerzini2004, FelmerMartinezTanak2008} 

\par
Since we are interested in radial solutions, the corresponding equation becomes
\begin{equation}\label{main1}
\left\{\begin{array}{ll}
-u''-\frac{N-1}ru'+u=u^p & \text{ in } (0,1)\\
u>0\text{ in } (0,1)\\
u'(0)=u'(1)=0.
\end{array}
\right.
\end{equation}
A crucial tool in our arguments is given by the nondegeneracy of the increasing (decreasing) solution in the annulus. We think that this result is interesting itself.\par
Our main result is the following,
\begin{theorem}\label{T}
Let $k>0$ be an integer. There exists $p(k)$ such that for any $p>p(k)$ problem \eqref{main} admits a radial solution $u_{p,k\text{layer}}(r)$ having exactly $k$ maximum points $\a_{1,p},\ldots,\a_{k,p}$.

Furthermore we have that
\begin{itemize}
\item[(i)] $(\a_{1,p},\ldots,\a_{k,p}) \to (\a_1,\ldots,\a_k)$ as $p\to\infty$ and $(\a_1,\ldots,\a_k)$ is a critical point of the function
\begin{equation}\label{eq:varphi_def}
\varphi(s_1,\ldots,s_k)=\inf \{ \|u\|_{H^1(B_1)}^2: \, u\in H^1_{rad}(B_1), \, u(s_1)=\ldots=u(s_k)=1 \},
\end{equation}
in the set $0<s_1<\ldots<s_k<1$;
\item[(ii)] $u_{p,k\text{layer}}(r)$ converges pointwise to $\sum_{j=1}^k A_j G(r,\a_j)$, where $(A_1,\ldots,A_k)$ is a solution of the system 
\end{itemize}
\begin{equation}\label{eq:system_A_j_def}
\sum_{j=1}^k A_j G(\a_i,\a_j)=1, \quad i=1,..,k.
\end{equation}
\end{theorem}


\subsection{Organization of the paper and main ideas}
In Section \ref{sec:limit_problem} we analyze in detail the limit problem $p=+\infty$. The radial increasing solution of the equation \eqref{main} in the annulus $B_b\setminus B_a$ (or in the ball when $a=0$) converges to the increasing function
\[
\frac{G(r,b)}{G(b,b)}, \quad r\in [a,b], \ a\geq0.
\]
Recall that $G(r,s)$ was defined above \eqref{eq:L_mathcal}. A decreasing solution of \eqref{main} exists only in the case of the annulus $a>0$, and converges to
\[
\frac{G(r,a)}{G(a,a)}, \quad r\in [a,b],\ a>0.
\]
By gluing an increasing solution and a decreasing solution, we construct a 1-layer solution in $B_b\setminus B_a$. This converges to 
\begin{equation}\label{eq:intro1}
\frac{G(r,\bar{s})}{G(\bar{s},\bar{s})}, \quad r\in [a,b],\ a\geq0,
\end{equation}
$\bar{s}$ being the unique point where the left derivative of the function is opposite to the right derivative.

Similarly, we study the limit problem of the $k$-layer solution. This is a combination of $k$ Green functions, with singular points being a critical point of the function $\varphi$ in \ref{eq:varphi_def}, and normalized with value 1 at the maximum points (see Theorem \ref{main1}). Again, the left and right derivatives are opposite at the maximum points. In order to prove the existence of a critical point of $\varphi$, we consider the juxtaposition of $k$ functions of the type \eqref{eq:intro1} and we prove, by a degree theorem, that there exists at least one configuration such that this juxtaposition is continuous.

In Section \ref{S2} we start the study of the problem $p<\infty$. We recall the variational characterization which ensures the existence of an increasing solution in the ball and in the annulus and of a decreasing solution in the annulus.

In Section \ref{S3} we prove that the increasing and decreasing solutions converge respectively to the two limit functions introduced above. The convergence is $C^1$ in the interior of the domain, but not on the boundary at the maximum point. In particular, we prove in Lemma \ref{lemma:u_p^p} that the value of the solution at the maximum point is asymptotically related to the value of the derivative of the limit profile.

In Section \ref{S4} we prove that the monotone solutions are nondegenerate. This is the most technical part of the paper and it is based on a blow-up argument inspired from \cite{G}. We present here in detail the proof of the uniqueness of the solution, which is very close to the proof of the nondegeneracy but presents some additional technical difficulties. The uniqueness and nondegeneracy ensure that the monotone solutions depend in a regular way on the boundary points $a$ and $b$. This is the basic tool to show the existence of a $k$-layer solution of \eqref{main} which bifurcates from $p=\infty$.

In section \ref{S7} we prove the existence of a 1-layer radial solution of \eqref{main}. We glue and increasing solution and a decreasing one. Thanks to the continuous dependence of the monotone solutions on the boundary points $a$ and $b$, we can show that there exists a continuous configuration. This function converges to \eqref{eq:intro1}. It is remarkable that the limit point $\bar{s}$ is a maximum point of the function $\varphi$ in \eqref{eq:phi_def} (whereas the monotone solutions are associated to minimum points of $\varphi$). 

In section \ref{S8} we construct the $k$-layer solution of \eqref{main}. This requires the additional property that the 1-layer solution is unique, both at the limit (see Lemma \ref{lemma:uniqueness_1_layer}) and for $p$ finite. To this aim we prove a stronger convergence result in Theorem \ref{thm:C1convergence}.

\subsection{Notation} We list below some notation used throughout the paper.
\begin{itemize}
\item[-] For $r>0$ we have $B_r=\{x\in\R^N: \, |x|<r\}$, $N\geq 3$; $|B_r|$ denotes the $N$-dimensional measure of $B_r$. For $0< r<R$, $B_R\setminus B_r =\{x\in\R^N: \, r<|x|<R\}$. In order to treat at the same time the case of the annulus and that of the ball, we will sometimes allow $r=0$ in the previous definition and use the convention that $B_R\setminus B_0=B_R$.
\item[-] $H^1_{rad}(B_r)$ denotes the Sobolev space of radial functions $H^1_{rad}(B_r)=\{u\in H^1(B_r):\, u=u(|x|)\}$. If $u\in H^1(B_r)$, $\|u\|^2_{H^1}=\int_{B_r}(|\nabla u|^2+u^2)\,dx$; if $u\in L^p(B_r)$, $1\leq p<\infty$, $\|u\|_p=\int_{B_r} |u|^p\,dx$; if $u\in L^\infty(B_r)$, $\|u\|_\infty=\text{ess\,sup}_{B_r}|u|$. Note that in the notation of the norms the domain is not specified and is taken as the domain of definition of the function.
\item[-] We denote by $u_p(r;\alpha,\beta)$ a solution of the problem \eqref{main} in the annulus $B_\beta\setminus B_\alpha$ (with Neumann b.c. on $\partial(B_\beta\setminus B_\alpha)$). The derivatives $u'_p(r;\alpha,\beta)$, $u''_p(r;\alpha,\beta)$, and so on, are taken with respect to the variable $r$.
\item[-] We adopt the standard notation $f(x)=o( g(x))$ as $x\to x_0$ if $\lim_{x\to x_0} f(x)/g(x)$ is zero, $f(x)=O(g(x))$ as $x\to x_0$ if $\limsup_{x\to x_0} |f(x)/g(x)|$ is finite, $f(x)\sim g(x)$ as $x\to x_0$ if $\lim_{x\to x_0} f(x)/g(x)$ is finite and different from zero.
\end{itemize}

\tableofcontents


\section{The limit problem}\label{sec:limit_problem}

\subsection{The 1-layer solution of the limit problem}

Finding a radial solution of 
\begin{equation}\label{main-annulus}
\left\{\begin{array}{ll}
-\Delta u+u=u^p & \text{ in } B_b\setminus B_a\\
u>0\\
\partial_\nu u=0 & \text{ on } \partial (B_b\setminus B_a),
\end{array}
\right.
\end{equation}
is easily done if $a>0$, whatever $p>1$, by minimizing the quotient 
\begin{equation}\label{eq:Q_p_def}
Q_{p,[a,b]}(u)=\frac{\|u\|^2_{H^1}}{\|u\|_{p+1}^2}, \quad u\in H^{1}_{rad}(B_b\setminus B_a)
\end{equation}
The limit problem as $p\to\infty$, namely minimizing 
\begin{equation}\label{eq:Q_infty_def}
Q_{\infty,[a,b]}(u) = \frac{\|u\|_{H^1}^2}{\|u\|_{\infty}^2}, \quad u\in H^{1}_{rad}(B_b\setminus B_a)
\end{equation}
was considered in \cite{BS} and \cite{GN}. In the study of this limit problem, it was shown that an important role is played by the function $\varphi : [a,b] \to \R$ defined by
\begin{equation}\label{eq:phi_def}
\varphi_{[a,b]} (s) = \inf_{\substack{u\in H^1_{rad} \\ u(s)\ne 0}}\frac{\|u\|_{H^1}^2}{u(s)^2}.
\end{equation}
This function $\varphi_{[a,b]}$ makes sense even if $a=0$, in which case we clearly have that the infimum is zero and not achieved for $s=a=0$, while it is achieved and not zero if $s>0$. 
For every $s\in\, ]a,b]$ when $a=0$ or $s\in\, [a,b]$ otherwise, there exists, up to normalization, a unique minimizer of \eqref{eq:phi_def}. Moreover, when $s\in\, ]a,b[$, this minimizer is given by the Green function, that we denote by $G_{[a,b]}(\cdot,s)$, associated to the operator
\[
\mathcal{L}: u\mapsto -u''-\frac{N-1}{r}u'+u
\]
for the boundary conditions $u'(a)=u'(b)=0$, i.e.
\begin{equation}\label{eq:Green_def}
\mathcal{L} G_{[a,b]}(r,s) =\delta_s \text{ for } r\in [a,b], \quad 
\frac{\partial G_{[a,b]}}{\partial r} (a,s)=\frac{\partial G_{[a,b]}}{\partial r} (b,s)=0.
\end{equation}
If $a>0$, the punctual limit of $G_{[a,b]}(r,s)$ as $s\to a$ is well defined and we denote it by $G_{[a,b]}(r,a)$. Analogously we will denote by $G_{[a,b]}(r,b)$ the limit of $G_{[a,b]}(r,s)$ as $s\to b$. Notice that  $G_{[a,b]}(r,a)$ and $G_{[a,b]}(r,b)$ satisfy \eqref{eq:Green_def} with only one boundary condition.

To simplify the notation we set
\[
G(r,s):=G_{[0,1]}(r,s).
\]
We recall in the next proposition some useful properties of $G(r,s)$.

\begin{proposition}[{\cite{Catrina2009,GN}}]\label{prop:xi_zeta}
There exist two positive linearly independent solutions $\zeta\in C^2((0,1])$ and $\xi\in C^2([0,1])$ of the equation $\mathcal{L}u=0$ in $(0,1)$ satisfying $\xi'(0)=\zeta'(1)=0$, 
\begin{equation}\label{eq:xi_zeta}
r^{N-1}(\xi'(r)\zeta(r)-\xi(r)\zeta'(r))= 1\quad\hbox{ for every }r\in (0,1], 
\end{equation}
and such that
\begin{equation}\label{eq:G}
G(r,s)=\left\{\begin{array}{ll}
s^{N-1}\xi(r)\zeta(s)\quad\text{for }r\leq s \\
s^{N-1}\xi(s)\zeta(r)\quad\text{for }r> s.
\end{array}\right.
\end{equation}
Moreover, $\xi$ is bounded and increasing in $[0,1]$, $\zeta$ is decreasing in $[0,1]$, and
\begin{equation}\label{eq:xi_zeta_limits1}
\displaystyle\lim_{r\rightarrow 0^+} \xi(r)=\frac{1}{N-2}
\end{equation}
\begin{equation}\label{eq:xi_zeta_limits2}
\displaystyle\lim_{r\rightarrow 0^+}r^{N-2}\z(r)=1,
\end{equation}
\begin{equation}\label{eq:xi_zeta_limits3}
\displaystyle\lim_{r\rightarrow 0^+}r^{N-1}\z'(r)=-(N-2)
\end{equation}
\end{proposition}
\begin{proof}
The existence of $\xi,\zeta$ satisfying \eqref{eq:xi_zeta} and \eqref{eq:G} is proved in \cite[Lemma 6.1]{GN}, based on the case of Dirichlet boundary conditions which can be found in \cite[Appendix]{Catrina2009}. Following step by step the last mentioned paper, one can also check that $\xi$ is bounded and that \eqref{eq:xi_zeta_limits1}- \eqref{eq:xi_zeta_limits3} hold. Finally, the monotonicity properties of $\xi$ and $\zeta$ follow by integrating the equation and using the boundary conditions $\xi'(0)=\zeta'(1)=0$ respectively.
\end{proof}

With $\xi$ and $\zeta$ as in Proposition \ref{prop:xi_zeta}, we define
\begin{equation}\label{eq:xi_zeta_annulus_def}
\begin{split}
& \xi_{[a,b]}(r)=\frac{\xi'(a)\zeta(r)-\xi(r)\zeta'(a)}{\sqrt{\xi'(a)\zeta'(b)-\xi'(b)\zeta'(a)}},
\  \zeta_{[a,b]}(r)=\frac{\xi'(b)\zeta(r)-\xi(r)\zeta'(b)}{\sqrt{\xi'(a)\zeta'(b)-\xi'(b)\zeta'(a)}} \quad \text{if } 0<a<b<1; \\
& \xi_{[0,b]}(r)=\frac{\xi(r)}{\xi'(b)},
\ \zeta_{[0,b]}(r)=\xi'(b)\zeta(r)-\xi(r)\zeta'(b) \quad \text{if } 0<b<1;\\
& \xi_{[a,1]}(r)=\xi'(a)\zeta(r)-\xi(r)\zeta'(a),
\ \zeta_{[a,1]}(r)=-\frac{\zeta(r)}{\zeta'(a)} \quad \text{if } 0<a<1.
\end{split}
\end{equation}

It follows from Proposition \ref{prop:xi_zeta} that these are the expressions of the Green function in the interval $[a,b]$. 

\begin{proposition}\label{prop:xi_zeta_annulus}
The functions $\xi_{[a,b]}$ and $\zeta_{[a,b]}$ are linearly independent positive solutions of $\mathcal{L}u=0$ in $(a,b)$, such that $\xi_{[a,b]}'(a)=\zeta_{[a,b]}'(b)=0$.
Furthermore, $\xi_{[a,b]}$ is increasing, $\zeta_{[a,b]}$ is decreasing,
\begin{equation}\label{eq:xi_zeta_annulus}
r^{N-1}(\xi_{[a,b]}'(r)\zeta_{[a,b]}(r)-\xi_{[a,b]}(r)\zeta_{[a,b]}'(r))= 1\quad\hbox{ for every }r\in (a,b], 
\end{equation}
and we have
\begin{equation}\label{eq:G_annulus}
G_{[a,b]}(r,s)=\left\{\begin{array}{ll}
s^{N-1}\xi_{[a,b]}(r)\zeta_{[a,b]}(s)\quad\text{for }r\leq s \\
s^{N-1}\xi_{[a,b]}(s)\zeta_{[a,b]}(r)\quad\text{for }r> s.
\end{array}\right.
\end{equation}

\end{proposition}
\begin{proof}
Let us show that $\xi'(a)\zeta'(b)-\xi'(b)\zeta'(a)\ne0$ for any $a,b\in(0,1)$. Indeed, for $r\in[a,b]$, we have
\begin{equation}
\chi_{[a,b]}(r):=\xi'(a)\zeta(r)-\xi(r)\zeta'(a)>0,
\end{equation}
since $\xi$ is increasing, $\zeta$ is decreasing and both functions are positive.
Moreover $\chi_{[a,b]}$ satisfies $\mathcal{L}(\chi_{[a,b]})=0$ in $[a,b]$ and $\chi_{[a,b]}'(a)=0$, which implies
\begin{equation}\label{eq:G2}
0<\chi_{[a,b]}'(b)=\xi'(a)\zeta'(b)-\xi'(b)\zeta'(a).
\end{equation}
The remaining properties can be proved by explicit computations.
\end{proof}

\begin{remark}
If $N=3$ the functions $\xi$ and $\zeta$ can be explicitly computed. In this case we have that
$\xi(r)=\frac{e^{r}-e^{-r}}{2r}$ and $\zeta(r)=\frac{e^r}r$. 
\end{remark}

The function $\varphi_{[a,b]}$ was shown in \cite{BS,GN} to have a global minimum at $a$ and a local minimum at $b$ (which is also consequence of the lemma below). We will recall in the next section that the local minimum in $b$ gives the limiting profile of the increasing radial solution for the original problem as $p\to\infty$ while if $a>0$, the global minimum at $a$ gives the limiting profile of the decreasing radial solution for the original problem as $p\to\infty$. We will build a third solution by gluing an increasing solution in a ball with a decreasing solution in an annulus. This is a 1-layer solution, having exactly one maximum point. The construction will use crucially the following fact.

\begin{lemma}\label{lemma:uniqueness_1_layer}
Let $0\leq a<b\leq 1$ and let $\varphi_{[a,b]}$ be as in \eqref{eq:phi_def}.
The function 
\[
s \mapsto \frac{\varphi_{[a,b]}'(s)}{s^{N-1}} \quad \text{is strictly decreasing}.
\]
Furthermore, its unique zero $\bar{s}$ satisfies $\bar{s}\in (a,b)$ and
\begin{equation}\label{eq:reflection_law}
\frac{\xi_{[a,b]}'(\bar{s})}{\xi_{[a,b]}(\bar{s})} + \frac{\zeta_{[a,b]}'(\bar{s})}{\zeta_{[a,b]}(\bar{s})} =0.
\end{equation}
\end{lemma}
\begin{proof}
Since the infimum in \eqref{eq:phi_def} is achieved by $G_{[a,b]}(\cdot,s)$, we have
\[
\varphi_{[a,b]}(s)=Q_\infty\left( \frac{G_{[a,b]}(\cdot,s)}{G_{[a,b]}(s,s)} \right).
\]
It is proved in \cite[Lemma 2.1]{GN} that
\[
Q_\infty\left( \frac{G_{[a,b]}(\cdot,s)}{G_{[a,b]}(s,s)} \right)=|\partial B_1|\frac{s^{N-1}}{G_{[a,b]}(s,s)}.
\]
Hence Proposition \ref{prop:xi_zeta_annulus} provides
\begin{equation}\label{eq:phi_rewritten}
\varphi_{[a,b]}(s)=\frac{|\partial B_1|}{\xi_{[a,b]}(s)\zeta_{[a,b]}(s)}
\end{equation}
We take the derivative of the last expression and we manipulate it by making use of \eqref{eq:xi_zeta_annulus} as follows
\begin{align}\label{eq:phi_derivative}
\frac{\varphi_{[a,b]}'(s)}{|\partial B_1|} & = - \frac{\xi_{[a,b]}'(s)\zeta_{[a,b]}(s)+\xi_{[a,b]}(s)\zeta_{[a,b]}'(s)}{\xi_{[a,b]}(s)^2\zeta_{[a,b]}(s)^2}\nonumber\\
& = - s^{N-1} \left\{ \left(\frac{\xi_{[a,b]}'(s)}{\xi_{[a,b]}(s)}\right)^2 -\left(\frac{\zeta_{[a,b]}'(s)}{\zeta_{[a,b]}(s)}\right)^2\right\}.
\end{align}
Let us study the monotonicity of the map $s\mapsto \xi_{[a,b]}'(s)/\xi_{[a,b]}(s)$. To this aim we multiply the equation $\mathcal{L} \xi_{[a,b]}=0$ by $r^{N-1}\xi_{[a,b]}'$ to get
\begin{equation}\label{eq:xi_xi'}
r^{N-1} \xi_{[a,b]}\xi_{[a,b]}'=r^{N-1} \xi_{[a,b]}''\xi_{[a,b]}'+(N-1) r^{N-2} (\xi_{[a,b]}')^2.
\end{equation}
Next we multiply the equation satisfied by $\xi_{[a,b]}'$ by $r^{N-1}\xi_{[a,b]}$ to obtain
\begin{equation}\label{eq:xi_xi'_bis}
-(r^{N-1}\xi_{[a,b]}''\xi_{[a,b]})' + r^{N-1} \xi_{[a,b]}''\xi_{[a,b]}'+(N-1) r^{N-3} \xi_{[a,b]}'\xi_{[a,b]}+r^{N-1} \xi_{[a,b]}'\xi_{[a,b]}=0.
\end{equation}
Replacing \eqref{eq:xi_xi'} into the last expression and integrating on $(a,s)$, we deduce that
\begin{equation}\label{eq:xi_xi'_tris}
s^{N-1} (\xi_{[a,b]}''(s)\xi_{[a,b]}(s)-\xi_{[a,b]}'(s)^2) = (N-1) \int_a^s r^{N-3} \xi_{[a,b]}'\xi_{[a,b]}\,dr.
\end{equation}
This implies
\begin{equation}\label{eq:xi_xi'_derivative}
\left(\frac{\xi_{[a,b]}'(s)}{\xi_{[a,b]}(s)}\right)'=\frac{\xi_{[a,b]}''(s)\xi_{[a,b]}(s)-\xi_{[a,b]}'(s)^2}{\xi_{[a,b]}(s)^2} >0.
\end{equation}
We preform the same computations with the function $\zeta_{[a,b]}$, but this time we integrate on $(s,b)$, leading to
\[
s^{N-1} (\zeta_{[a,b]}''(s)\zeta_{[a,b]}(s)-\zeta_{[a,b]}'(s)^2) = 
\zeta_{[a,b]}''(b)\zeta_{[a,b]}(b) - (N-1) \int_s^b r^{N-3} \zeta_{[a,b]}'\zeta_{[a,b]}\,dr.
\]
Since $\zeta_{[a,b]}$ is decreasing and solves $\mathcal{L}\zeta_{[a,b]}=0$, $\zeta_{[a,b]}'(b)=0$, we deduce that the previous expression is also positive. Taking again into account the monotonicity of the maps $\xi_{[a,b]}$ and $\zeta_{[a,b]}$, we conclude that
\begin{equation}\label{eq:zeta_zeta'_derivative}
s \mapsto \left(\frac{\xi_{[a,b]}'(s)}{\xi_{[a,b]}(s)}\right)^2 \text{ is increasing}, \quad
s \mapsto \left(\frac{\zeta_{[a,b]}'(s)}{\zeta_{[a,b]}(s)}\right)^2 \text{ is decreasing}.
\end{equation}
By combining this with \eqref{eq:phi_derivative}, we deduce that $\varphi_{[a,b]}'(s)/s^{N-1}$ is decreasing and that relation \eqref{eq:reflection_law} holds at the unique zero $\bar{s}$.

It only remains to show that $\bar{s}$ is an interior point. If $a\neq0$, we see from \eqref{eq:phi_derivative} that $\varphi_{[a,b]}'(a)>0$ and $\varphi_{[a,b]}'(b)<0$ so that the claim obviously follows. When $a=0$, we still have $\varphi_{[a,b]}'(b)<0$ and moreover
\[
\lim_{s\to0} \frac{\varphi_{[0,b]}'(s)}{s^{N-1}} = \lim_{s\to0} \frac{(N-2)^2}{s^2} =+\infty,
\]
as follows from \eqref{eq:xi_zeta_annulus_def} and Proposition \ref{prop:xi_zeta}.
\end{proof}

\begin{remark}
The formula \eqref{eq:reflection_law} that defines $\bar s$ is equivalent to
$$
\left. \frac{d}{dr}\left(\frac{G_{[a,b]}(r,r)}{r^{N-1}}\right)\right|_{r=\bar s}=0,
$$
which means that $\bar s$ is a critical point of the weighted Robin function associated to $G_{[a,b]}$. So one deduces the following statement from Lemma \ref{lemma:uniqueness_1_layer}: for any $0\leq a<b\leq 1$, the weighted Robin function associated to $G_{[a,b]}$ has a unique  interior critical point. 
\end{remark}

Since Lemma \ref{lemma:uniqueness_1_layer} provides the uniqueness of $\bar s$, we can define the map 
$$\bar s : (a,b)\mapsto \bar s(a,b),$$
which is defined in the set $\{0< a < 1,\ a<b\le 1\}$. Similarly, when we are working in the annulus, this is a function of one variable $\bar{s}(0,b)$ defined in $\{0<b\leq1\}$. The monotonicity proved in Lemma \ref{lemma:uniqueness_1_layer} implies that this map is smooth. 

\begin{lemma}\label{lemma:alpha_regular}
The map $(a,b)\mapsto \bar s(a,b)$ is of class $C^{1}$ in the set $\{0< a < 1,\ a<b\le 1\}$. Analogously, $\bar{s}(0,b)$ is of class $C^1$ in $\{0<b\leq1\}$
\end{lemma}
\begin{proof} It follows from Lemma \ref{lemma:uniqueness_1_layer} that 
$\bar s$ is implicitely defined by the equation
\[
0=F(a,b,s)=
\left(\frac{\xi_{[a,b]}'(s)}{\xi_{[a,b]}(s)}\right)^2 -\left(\frac{\zeta_{[a,b]}'(s)}{\zeta_{[a,b]}(s)}\right)^2.
\]
The definitions of $\xi_{[a,b]}$ and $\zeta_{[a,b]}$ imply that $F$ is smooth. Let $0<a_{0}<b_{0}<1$ and $s_0=\bar{s}(a_0,b_0))$. Since by \eqref{eq:zeta_zeta'_derivative} we have $\partial F/\partial s(a_0,b_0,s_0)>0$, the Implicit Function Theorem applies and $\bar s$ is a $C^{1}$ function of $(a,b)$ in a neighborhood of $(a_0,b_0)$. This holds for every $0<a_{0}<b_{0}<1$.
When $b_{0}=1$ we argue in the same way in a left neighborhood of $b_{0}$. In the case of the ball $a_0=0$ we  can proceed similarly.
\end{proof}

Next we study the behaviour of $\bar s$ when $b\to 0$. 

\begin{lemma}\label{lemma:alpha_asimpt_beta}
Let $0\leq a<b\leq 1$. We have 
\begin{equation}\label{c20} 
\bar s \sim b  \quad\text{as } b\to 0^+. 
\end{equation}
\end{lemma}
\begin{proof}
Let us first consider the case $a=0$. We use \eqref{eq:xi_zeta_annulus_def} to rewrite
the equation \eqref{eq:reflection_law} as
\begin{equation}\label{c21}
\frac{ \xi'(\bar s)}{\xi(\bar s)} 
+\frac{\xi'(b)\zeta'(\bar s)-\xi'(\bar s)\zeta'(b)}{\xi'(b)\zeta(\bar s)-\xi(\bar s)\zeta'(b)}=0.
\end{equation}
Both $b$ and $\bar s$ converge to zero, hence we can replace in the previous expression the following asymptotic developments, which are deduced from Proposition \ref{prop:xi_zeta}:
\begin{equation}\label{eq:asymptotics_xi_zeta}
\begin{split}
\xi(r)=\frac{1}{N-2}+o(r), \quad \xi'(r)=\xi''(0)r+o(r), \quad\text{as } r\to0\\
\zeta(r)=\frac{1}{r^{N-2}}+o\left(\frac{1}{r^{N-2}}\right), \quad 
\zeta'(r)=-\frac{N-2}{r^{N-1}}+o\left(\frac{1}{r^{N-1}}\right) \quad\text{as } r\to0.
\end{split}
\end{equation}
Using these asymptotics in \eqref{c21}, we infer that
\[
2\bar s^N-b^N+o(\bar s^N)+o(b^N) =0,
\]
which implies that $\bar s \sim b$ as $b\to0$.

\medbreak

Assume now that $a>0$. We rewrite \eqref{eq:reflection_law} more explicitely as
\begin{equation}\label{eq:reflection_law3}
\frac{\xi'(a)\zeta'(\bar s)-\xi'(\bar s)\zeta'(a)}{\xi'(a)\zeta(\bar s)-\xi(\bar s)\zeta'(a)}
+ \frac{\xi'(b)\zeta'(\bar s)-\xi'(\bar s)\zeta'(b)}{\xi'(b)\zeta(\bar s)-\xi(\bar s)\zeta'(b)} =0.
\end{equation}
Since now $a,\bar s,b\to0$, we can use again \eqref{eq:asymptotics_xi_zeta} to obtain
\begin{equation}\label{eq:a_j_sim_b_j}
2\bar s^{2N-2}-(a^N+b^N)\bar s^{N-2}-2\xi''(0)a^N b^N
+ o(\bar s^{2N-2})+ o((a^N+b^N)\bar s^{N-2}) + o(a^N b^N) =0.
\end{equation}
Here we have to distinguish two cases. If $a\sim b$ as $b\to0$, then \eqref{eq:a_j_sim_b_j} writes
\[
\bar s^{2N-2}-C_1\bar s^{N-2}b^N-C_2b^{2N}+o(\bar s^{2N-2})+o(\bar s^{N-2}b^N)+o(b^{2N})=0.
\]
for some $C_1,C_2>0$. This implies $\bar s \sim b$. Indeed, if by contradiction $\bar s=o(b)$, then also $\bar s^{2N-2} =o(\bar s^{N-2} b^N)$, and we would obtain
\[
-C_1\bar s^{N-2}-C_2b^N+o(\bar s^{N-2})+o(b^N)=0, \quad C_1>0,C_2>0,
\]
which is not possible. When $a=o(b)$, \eqref{eq:a_j_sim_b_j} yields
\[
2\bar s^{2N-2}-b^N\bar s^{N-2}+o(\bar s^{2N-2})+o(b^N\bar s^{N-2})+o(b^{2N})=0.
\]
Again, if $\bar s=o(b)$, then we obtain $-\bar s^{N-2}+o(\bar s^{N-2})+o(b^N)=0$ and since this is not possible, we conclude that $\bar s \sim b$ also in this case.
\end{proof}

In the next lemma, we show that the distance from $\bar s$ to the extrema of the interval only depends on the length of the interval.

\begin{lemma}\label{lemma:alpha_not_going_beta}
Let $0\leq a<b\leq 1$. For every $\ep>0$, there exists $\delta>0$ such that
if $
b-a>\ep$, then  
$b-\bar s>\delta$  and $\bar s -a>\delta$.
\end{lemma}
\begin{proof}
We argue by contradiction. Let $(a_{n})_{n}$ and $(b_{n})_{n}$ be such that $b_{n}-a_{n}>\varepsilon$ and $(\bar s_{n})_{n}$  We have to distinguish three cases

Assume first $a>0$. Suppose by contradiction that there exist sequences $\a_{j}^{(n)},\b_{j-1}^{(n)}$ such that $|\a_{j}^{(n)}-\b_{j-1}^{(n)}|\to0$ as $n\to\infty$. Replacing in \eqref{eq:reflection_law3}, we obtain
\[
\xi'(\b_{j}^{(n)})\zeta'(\b_{j-1}^{(n)}) -\xi'(\b_{j-1}^{(n)})\zeta'(\b_{j}^{(n)}) =0,
\]
which contradicts \eqref{eq:G2} since $|\b_{j-1}^{(n)}-\b_j^{(n)}|>\ep$ for every $n$. The case $|\a_{j}^{(n)}-\b_{j}^{(n)}|\to0$ can be ruled out in the same way. The cases $j=1$ and $j=k$ can be proved in a similar way, by exploiting the suitable definitions in \eqref{eq:xi_zeta_annulus_def}.
\end{proof}

\bigbreak

We will see that the point $\bar{s}$ defined in the previous lemma is the limit (as $p\to\infty$) of the maximum points of the 1-layer solutions of \eqref{main}. Therefore we give the following definition of 1-layer solution of the limit problem in an interval $[\b_{j-1},\b_j]$.
For $k\in\mathbb{N}_0$, let
\[
0=\b_0<\b_1<\ldots<\b_{k-1}<\b_k=1.
\] 
We denote by $\a_j:=\a_j(\b_{j-1},b_j)$ the unique point satisfying \eqref{eq:reflection_law} in the interval $[\b_{j-1},\b_j]$, namely
\begin{equation}\label{eq:reflection_law2}
\frac{\xi_{[\b_{j-1},\b_j]}'(\a_j)}{\xi_{[\b_{j-1},\b_j]}(\a_j)} 
+ \frac{\zeta_{[\b_{j-1},\b_j]}'(\a_j)}{\zeta_{[\b_{j-1},\b_j]}(\a_j)} =0, 
\quad j=1,\ldots k.
\end{equation}


\begin{definition}\label{def:1_layer_limit}
We refer to the function
\begin{equation}\label{eq:1layer_def}
u_{\infty,1\text{-layer}}(r;\b_{j-1},\b_{j}):=
\frac{G_{[\b_{j-1},\b_{j}]}(r,\a_j)}{G_{[\b_{j-1},\b_{j}]}(\a_j,\a_j)}
\end{equation}
as the \emph{1-layer solution of the limit problem} in the interval $[\b_{j-1},\b_{j}]$.
\end{definition}
When we do not need to emphasize the interval of definition, we write $u_{\infty,1\text{-layer}}(r)$ to shorten the notations. 
Observe that  \eqref{eq:reflection_law2} shows that $u_{\infty,1\text{-layer}}$ satisfies a \emph{reflection law} at $\a_j$: the right and left derivatives are opposite, namely
\[
\lim_{\ep\to0^-} \frac{u_{\infty,1\text{-layer}}(\a_j+\ep)-u_{\infty,1\text{-layer}}(\a_j)}{\ep}
=- \lim_{\ep\to0^+} \frac{u_{\infty,1\text{-layer}}(\a_j+\ep)-u_{\infty,1\text{-layer}}(\a_j)}{\ep}.
\]

\subsection{The k-layer solution of the limit problem}

In order to produce a \emph{$k$-layer solution of the limit problem}, we glue together $k$ 1-layer solutions. For $k\in\N_0$, let 
\begin{equation}\label{eq:T_def}
T=\{ (\b_1,\ldots,\b_{k-1})\in\R^{k-1}: \, 0=\b_0<\b_1<\b_2<\ldots<\b_{k-1}<\b_k=1 \}.
\end{equation}
The existence of a continuous configuration will follow from a degree argument, applied to the map $M_\infty=(M_\infty^{(1)},\ldots,M_\infty^{(k-1)}):T\to\R^{k-1}$, defined as
\begin{equation}\label{eq:M_infty_def}
\begin{split}
M_\infty^{(j)}(\b_1,\ldots,\b_{k-1})
&= u_{\infty,1\text{-layer}}(\b_j;\b_j,\b_{j+1})-u_{\infty,1\text{-layer}}(\b_j;\b_{j-1},\b_{j}) \\
&= \frac{\xi_{[\b_j,\b_{j+1}]}(\b_j)}{\xi_{[\b_j,\b_{j+1}]}(\a_{j+1})} - 
\frac{\zeta_{[\b_{j-1},\b_{j}]}(\b_j)}{\zeta_{[\b_{j-1},\b_{j}]}(\a_{j})} \\
&=\frac{1}{\b_j^{N-1}} \left\{ \frac{1}{\xi'(\b_j)\zeta(\a_{j+1})-\xi(\a_{j+1})\zeta'(\b_j)} -
\frac{1}{\xi'(\b_j)\zeta(\a_{j})-\xi(\a_{j})\zeta'(\b_j)} \right\},
\end{split}
\end{equation}
for $j=1,\ldots,k-1$, where $\a_j=\a_j(\b_{j-1},\b_j)$ are defined in \eqref{eq:reflection_law2}. We notice that $M_\infty^{(j)}$ depends only on $\b_{j-1},\b_j,\b_{j+1}$.

In order to study the degree of $M_\infty$ in $T$, we need to evaluate it on $\partial T$, given by the union of $k$ sets:
\begin{equation}\label{eq:partial_T}
\partial T= \cup_{j=1}^k (\partial T)_j, \quad\text{with}\quad
(\partial T)_j=\{\b_0\leq\b_1\leq\ldots\leq\b_{j-1}=\b_j\leq\ldots\leq\b_{k-1}\leq\b_k\}.
\end{equation}
Using a standard notation, we denote by $\bar{T}$ the closure of $T$, that is to say $\overline{T}=T\cup\partial T$.

\begin{lemma}\label{lemma:boundary_T_1}
There exists $\bar{\ep}$ such that for every $0<\ep<\bar{\ep}$ there exists a constant $C(\ep)>0$ (depending only on $\ep$) such that, for every $j=1,\ldots,k-1$, we have
\begin{equation*}
M_\infty^{(j)}(\b_1,\ldots,\b_{k-1})<-C(\ep) \quad 
\text{for } (\b_1,\ldots,\b_{k-1})\in \overline{T} \text{ with } \b_j\leq\ep\leq\b_{j+1}.
\end{equation*}
\end{lemma}
\begin{proof}
Fix any $j=1,\ldots,k-1$ and any $\ep\in(0,1)$. We compute the limit of $M_\infty^{(j)}$ as $\beta_j\to0$ and $\b_{j+1}>\ep$.
Consider the definition of $M_\infty^{(j)}$ in \eqref{eq:M_infty_def}. Since $\a_j,\b_j\to0$, we can replace the developments \eqref{eq:asymptotics_xi_zeta}. Moreover, thanks to Lemma \ref{lemma:alpha_not_going_beta}, for every fix $\b_{j+1}>\ep>0$, we have that $\a_{j+1}\to\bar{\alpha}_{j+1}$, with $\bar{\alpha}_{j+1}>\delta>0$ and $\delta=\delta(\ep)$ independent of $\b_{j+1}$. Therefore
\begin{multline}\label{eq:boundary_T_1a}
\lim_{\substack{\b_j\to0\\\b_{j+1}>\ep}} M_\infty^{(j)}(\b_1,\ldots,\b_{k-1})=\\
\lim_{\b_j\to0} \left\{\frac{1}{\xi''(0)\zeta(\bar{\a}_{j+1})\b_j^N+(N-2)\xi(\bar{\a}_{j+1})+o(1)}
-\frac{1}{\xi''(0)\frac{\b_j^N}{\a_j^{N-2}}+1+o(\frac{\b_j^N}{\a_j^{N-2}})+o(1)} \right\}.
\end{multline}
By Lemma \ref{lemma:alpha_asimpt_beta} we have that $\a_j\sim\b_j$, so that
\begin{equation}\label{eq:boundary_T_1b}
\lim_{\substack{\b_j\to0\\\b_{j+1}>\ep}} M_\infty^{(j)}(\b_1,\ldots,\b_{k-1})=
\frac1{(N-2)\xi(\bar{\a}_{j+1})}-1 <-C(\ep),
\end{equation}
where in the last step we used the fact that $\xi(0)=1/(N-2)$ and $\xi(\bar{\a}_{j+1})>1/(N-2)+C(\delta)$ for $\bar{\alpha}_{j+1}>\delta>0$.\par
\end{proof}

\begin{lemma}\label{lemma:M_infty_continuous}
$M_\infty$ is continuous in $T$ and can be extended continuously on $\overline{T}$.
\end{lemma}
\begin{proof}
Thanks to Lemma \ref{lemma:alpha_regular}, $M_\infty$ is continuous in $T$. 
Calculations similar to the ones in \eqref{eq:boundary_T_1a}, \eqref{eq:boundary_T_1b} show that $M_\infty$ can be extended continuously on $\overline{T}$
\end{proof}

\begin{lemma}\label{lemma:boundary_T_k}
There exists $\bar{\ep}$ such that for every $0<\ep<\bar{\ep}$ there exists a constant $C(\ep)>0$ such that, for every $j=1,\ldots,k-1$, we have
\begin{equation*}
M_\infty^{(j)}(\b_1,\ldots,\b_{k-1})>C(\ep) \quad 
\text{for } (\b_1,\ldots,\b_{k-1})\in \overline{T} \text{ with } \b_{j-1}\leq 1-\ep\leq\b_j. 
\end{equation*}
\end{lemma}
\begin{proof}
We proceed similarly to the previous lemma. 
 Now we use the fact that $\zeta'(1)=0$ and that $\a_{j}\to\bar{\a}_{j}<1-\delta$, with $\delta>0$, by Lemma \ref{lemma:alpha_not_going_beta}. We have
\[
\lim_{\substack{\b_{k-1}\to1\\\b_{j-1}<1-\ep}} M_\infty^{(j)}(\b_1,\ldots,\b_{k-1})=
1-\frac{1}{\xi'(1)\zeta(\bar{\a}_{j})} > C(\ep),
\]
since $\xi'(1)\zeta(\bar{\a}_{j})>\xi'(1)\zeta(1)+C(\delta)$ for $\bar{\a}_{j}<1-\delta$, and $\xi'(1)\zeta(1)=1$. Again, by taking $\ep$ sufficiently small, the statement follows.
\end{proof}

\begin{lemma}\label{lemma:collapse_points}
There exists $\bar{\ep}$ such that for every $0<\ep<\bar{\ep}$ there exists a constant $C(\ep)>0$ such that
\begin{itemize}
\item[(i)] for every $j=2,\ldots,k$ and $l=1,\ldots,j-1$ we have
\begin{equation}
M_\infty^{(l)}(\b_1,\ldots,\b_{k-1})>C(\ep) 
\quad \text{for } (\b_1,\ldots,\b_{k-1})\in \overline{T} \text{ with } \b_{l-1}\leq\b_{j}-\ep\leq\b_l;
\end{equation}
\item[(ii)] for every $j=0,\ldots,k-2$ and $l=j+1,\ldots,k-1$ we have
\begin{equation}
M_\infty^{(l)}(\b_1,\ldots,\b_{k-1})<-C(\ep) 
\quad \text{for } (\b_1,\ldots,\b_{k-1})\in \overline{T} \text{ with } \b_{l}\leq\b_{j}+\ep\leq\b_{l+1}.
\end{equation}
\end{itemize}
\end{lemma}
\begin{proof}
We compute the limit of $M_\infty^{(l)}$ as $\b_{l}\to\b_j$ and $\b_{l-1}<\b_j-\ep$. To this aim, consider the definition of $M_\infty^{(l)}$ in \eqref{eq:M_infty_def} and notice that both $\a_{l+1}\to\b_j$ and $\b_{l+1}\to\b_j$ as $\b_{l}\to\b_j$. For every fix $\b_{l-1}$, denote by $\bar{\a}_{l}$ the limit of $\a_{l}$ as $\b_{l}\to\b_j$. We obtain
\[
\lim_{\substack{\b_{l}\to\b_j\\\b_{l-1}<\b_j-\ep}} M_\infty^{(l)}(\b_1,\ldots,\b_{k-1})
=1-\frac{\zeta_{[\b_{l-1},\b_j]}(\b_j)}{\zeta_{[\b_{l-1},\b_j]}(\bar{\a}_{l})}
\]
By Lemma \ref{lemma:alpha_not_going_beta} there exists $\delta$ independent of $\b_{l-1}$ such that $\bar{\a}_{l}<\b_j-\delta$. Since $\zeta_{[\b_{l-1},\b_j]}$ is decreasing, we conclude that the previous quantity is larger than a strictly positive constant which depends only on $\ep$. By taking $\ep$ sufficiently small, the statement follows. Similarly we have
\[
\lim_{\substack{\b_{l}\to\b_j\\\b_{l+1}>\b_j+\ep}} M_\infty^{(l)}(\b_1,\ldots,\b_{k-1})
=\frac{\xi_{[\b_{j},\b_{l+1}]}(\b_j)}{\xi_{[\b_{j},\b_{l+1}]}(\bar{\a}_{l+1})} -1 <-C(\ep). \qedhere
\]
\end{proof}

\begin{theorem}\label{thm:k-layer-limit}
Let $k\in\N_0$. There exists a configuration $0=\beta_{0}<\beta_{1}<\ldots<\beta_{k-1}<\b_k=1$ such that the function
\begin{equation}
u_{\infty,k\text{layer}}(r):=u_{\infty,1\text{-layer}}(r;\b_{j-1},\b_j) \quad\text{ for } r\in [\b_{j-1},\b_j), \ j=0,\ldots,k, 
\end{equation}
is continuous. In addition, the $\b_j$ satisfy
\begin{equation}\label{eq:b_j_def}
\frac{\xi'(\b_j)}{\zeta'(\b_j)}=\frac{\xi(\a_{j+1})-\xi(\a_j)}{\zeta(\a_{j+1})-\zeta(\a_j)}, \quad j=1,\ldots, k-1.
\end{equation}
\end{theorem}
\begin{proof}
For $P=(P_1,\ldots,P_{k-1})\in T$ to be chosen later, let us introduce the operator
\[
(I-P)(\b_1,\ldots,\b_{k-1})=\left(\beta_1-P_1,\ldots,\beta_{k-1}-P_{k-1}\right).
\]
We want to show that the homotopy $H=(H^{(1)},\ldots,H^{(k-1)})$ defined by
\begin{equation}\label{eq:homotopy}
H(t,\b_1,\ldots,\b_{k-1})=tM_\infty(\b_1,\ldots,\b_{k-1})+(1-t)(I-P)(\b_1,\ldots,\b_{k-1})
\end{equation}
satisfies
\begin{equation}\label{1}
H(t,\b_1,\ldots,\b_{k-1})\ne0\quad\text{for every } t\in[0,1]\hbox{ and }(\b_1,\ldots,\b_{k-1})\in\partial T.
\end{equation}
Here $M_\infty$ is extended to $\partial T$ thanks to Lemma \ref{lemma:M_infty_continuous}.

In the following take $\ep<\bar{\ep}/2$, with $\bar{\ep}$ such that the statements of Lemmas \ref{lemma:boundary_T_1}, \ref{lemma:boundary_T_k} and \ref{lemma:collapse_points} hold true.

Let us first consider $H$ on $(\partial T)_1$, as defined in \eqref{eq:partial_T}. We write
\[
(\partial T)_1=\cup_{j=1}^{k-1}(\partial T)_{1,j}, \quad
(\partial T)_{1,j} :=\{ (\b_1,\ldots,\b_{k-1})\in (\partial T)_1: \, \b_j\leq\ep\leq\b_{j+1}\}
\]
By Lemma \ref{lemma:boundary_T_1} there exists $C>0$ such that
\[
H^{(j)}(t,\b_1,\ldots,\b_{k-1})<-tC+(1-t)(\ep-P_j) \quad\text{on } (\partial T)_{1,j},
\]
for every $j=1,\ldots,k-1$. This quantity is negative for every $t\in[0,1]$ provided that $P_j>\ep$ for every $j$.

Let us consider $H$ on $(\partial T)_k$. We write
\[
(\partial T)_k=\cup_{j=1}^{k-1}(\partial T)_{k,j}, \quad
(\partial T)_{k,j} :=\{ (\b_1,\ldots,\b_{k-1})\in (\partial T)_k: \, \b_{j-1}\leq1-\ep\leq\b_{j} \}.
\]
Then, by Lemma \ref{lemma:boundary_T_k},
\[
H^{(j)}(t,\b_1,\ldots,\b_{k-1})>tC+(1-t)(1-\ep-P_j) \quad\text{on } (\partial T)_{k,j},
\]
and this quantity is positive for every $t\in[0,1]$ provided that $P_j<1-\ep$ for every $j$.

Finally, let us consider $H$ on $(\partial T)_j$, for a fix $j=1,\ldots,k-1$. We define
\[
(\partial T)_{j,-1}^-=\{(\b_1,\ldots,\b_{k-1}) \in (\partial T)_j :\, \b_{j-1}=\b_j \leq\ep \},
\]
\[
(\partial T)_{j,l}^-=\{(\b_1,\ldots,\b_{k-1}) \in (\partial T)_j :\, \b_l\leq\b_{j-1}-\ep\leq\b_{l+1} \},
\]
for $l=0,\ldots,j-2$, and
\[
(\partial T)_{j,m}^+=\{(\b_1,\ldots,\b_{k-1}) \in (\partial T)_j :\, \b_{m-1}\leq \b_j+\ep\leq \b_m \},
\]
\[
(\partial T)_{j,k+1}^+=\{(\b_1,\ldots,\b_{k-1}) \in (\partial T)_j :\, \b_{j-1}=\b_j\geq1-\ep \},
\]
for $m=j+1,\ldots,k$, so that
\[
(\partial T)_{j}=\cup_{m=j+1}^k\left( (\partial T)_{j,-1}^-\cap (\partial T)_{j,m}^+ \right)
\cup_{l=0}^{j-2}\cup_{m=j+1}^{k+1}\left( (\partial T)_{j,l}^-\cap(\partial T)_{j,m}^+ \right).
\]
Let us show that on each piece of this decomposition at least one component  of $H$ does not vanish.

On $(\partial T)_{j,-1}^-\cap (\partial T)_{j,m}^+$, $m=j+1,\ldots,k$, we have $\b_{m-1}\leq \b_j+\ep\leq 2\ep$ and $\b_m\geq \b_j+\ep\geq\ep$.  Lemma \ref{lemma:boundary_T_1} implies
\[
M_\infty^{(m-1)}(\b_1,\ldots,\b_{k-1})<-C 
\quad \text{on } (\partial T)_{j,-1}^-\cap (\partial T)_{j,m}^+ ,
\]
so that
\[
H^{(m-1)}(t,\b_1,\ldots,\b_{k-1})<-Ct+(1-t) (2\ep-P_{m-1}) <0
\quad \text{on } (\partial T)_{j,-1}^-\cap (\partial T)_{j,m}^+,
\]
for every $m=j+1,\ldots,k$.

Next consider $(\partial T)_{j,l}^-\cap(\partial T)_{j,m}^+$, for $l=0,\ldots,j-2$ and $m=j+1,\ldots,k$. Lemma \ref{lemma:collapse_points} implies that
\[
M_\infty^{(l+1)}>C \text{ and } M_\infty^{(m-1)}<-C 
\quad\text{on } (\partial T)_{j,l}^-\cap(\partial T)_{j,m}^+,
\]
hence
\[
H^{(l+1)}>Ct+(1-t) (\b_{l+1}-P_{l+1}),\quad
H^{(m-1)}<-Ct+(1-t) (\b_{m-1}-P_{m-1})
\]
on $(\partial T)_{j,l}^-\cap(\partial T)_{j,m}^+$. Suppose by contradiction that both $H^{(l+1)}$ and $H^{(m-1)}$ vanish on $(\partial T)_{j,l}^-\cap(\partial T)_{j,m}^+$, then
\[
\b_{j-1}-\ep\leq \b_{l+1}<P_{l+1}<P_{m-1}<\b_{m-1}\leq\b_j+\ep=\b_{j-1}+\ep,
\]
which is not possible provided that $|P_{l+1}-P_{m-1}|>2\ep$.

Finally on $(\partial T)_{j,l}^-\cap(\partial T)_{j,k+1}^+$, $l=0,\ldots,j-2$, we have $\b_l\leq\b_{j-1}-\ep\leq1-\ep$ and $\b_{l+1}\geq\b_{j-1}-\ep\geq1-2\ep$. Lemma \ref{lemma:boundary_T_k} implies
\[
H^{(l+1)}> Ct+(1-t) (1-2\ep-P_{l+1}) >0
\quad \text{on } (\partial T)_{j,l}^-\cap(\partial T)_{j,k+1}^+,
\]
for every $l=0,\ldots,j-2$.

By \eqref{eq:homotopy} we get that
\begin{equation}\label{6}
deg\left(M_\infty(\b_1,\ldots,\b_{k-1}), T,0\right)=deg\left(I-P,T,0\right)=1.
\end{equation}
Then the equation
\begin{equation}\label{7}
M_\infty(\b_1,\ldots,\b_{k-1})=0
\end{equation}
admits at least a solution in $T$.
\end{proof}

\begin{proposition}
We have
\begin{equation}
u_{\infty,k-\text{layer}}=\sum_{j=1}^k A_j G(r,\a_j),
\end{equation}
where $(\a_1,\ldots,\a_k)$ is a critical point of the function $\varphi$ defined in \eqref{eq:varphi_def} and $(A_1,\ldots,A_k)$ is a solution of the system \eqref{eq:system_A_j_def}.
\end{proposition}
\begin{proof}
By construction, $u_{\infty,k-\text{layer}}$ is the juxtaposition of $k$ 1-layer solutions $u_{\infty,1\text{-layer}}(r;\b_{j-1},\b_j)$ as defined in \eqref{eq:1layer_def}. The $\b_j$ are such that the juxtaposition is continuous, that is to say \eqref{eq:b_j_def} holds. Recall that each 1-layer solution attains it maximum value 1 at $r=\a_j$, with $\a_j$ satisfying \eqref{eq:reflection_law2}, therefore $(A_1,\ldots,A_k)$ solves the system \eqref{eq:system_A_j_def}. We only have to prove that $(\a_1,\ldots,\a_k)$ is a critical point of $\varphi$.

Let us write relation \eqref{eq:reflection_law2} more explicitely for $\b_j$ satisfying \eqref{eq:b_j_def}:
\begin{equation}\label{eq:reflection_law_continuous1}
\frac{\xi'(\a_1)}{\xi(\a_1)} + \frac{\zeta'(\a_1)[\xi(\a_2)-\xi(\a_1)]-\xi'(\a_1)[\zeta(\a_2)-\zeta(\a_1)]}{\xi(\a_2)\zeta(\a_1)-\zeta(\a_2)\xi(\a_1)}=0, 
\end{equation}
\begin{multline}\label{eq:reflection_law_continuousj}
\frac{\zeta'(\a_j)[\xi(\a_j)-\xi(\a_{j-1})]-\xi'(\a_j)[\zeta(\a_j)-\zeta(\a_{j-1})]}{\xi(\a_j)\zeta(\a_{j-1})-\zeta(\a_j)\xi(\a_{j-1})} \\
+\frac{\zeta'(\a_j)[\xi(\a_{j+1})-\xi(\a_{j})]-\xi'(\a_j)[\zeta(\a_{j+1})-\zeta(\a_{j})]}{\xi(\a_{j+1})\zeta(\a_{j})-\zeta(\a_{j+1})\xi(\a_{j})}=0 \quad j=2,\ldots,k-1, \\
\end{multline}
\begin{equation}\label{eq:reflection_law_continuousk}
\frac{\zeta'(\a_k)[\xi(\a_k)-\xi(\a_{k-1})]-\xi'(\a_k)[\zeta(\a_k)-\zeta(\a_{k-1})]}{\xi(\a_k)\zeta(\a_{k-1})-\zeta(\a_k)\xi(\a_{k-1})} + \frac{\zeta'(\a_k)}{\zeta(\a_k)}=0.
\end{equation}
We have to prove that
\[
\varphi'(\a_1,\ldots,\a_k)=0 \quad\text{if and only if}\quad  (\a_1,\ldots,\a_k)\text{ satisfys } \eqref{eq:reflection_law_continuous1}-\eqref{eq:reflection_law_continuousk}
\]
 
We have
\begin{equation}
\begin{split}
\varphi(s_1,\ldots,s_k)
& =|\partial B_1| \cdot \sum_{j=1}^k s_j^{N-1} 
\left[ u_{\infty,+}'(s_j;\b_{j-1},s_j)-u_{\infty,-}'(s_j;s_j,\b_j) \right] \\
& =|\partial B_1| \cdot \sum_{j=1}^k s_j^{N-1} \left[ 
\frac{\xi_{[\b_{j-1},\b_j]}'(s_j)}{\xi_{[\b_{j-1},\b_j]}(s_j)} 
-\frac{\zeta_{[\b_{j-1},\b_j]}'(s_j)}{\zeta_{[\b_{j-1},\b_j]}(s_j)}
\right]=:|\partial B_1| \cdot \sum_{j=1}^k \Phi_j,
\end{split}
\end{equation}
with the $\b_j$ satisfying (see \eqref{eq:b_j_def})
\begin{equation}\label{eq:b_j_def2}
\frac{\xi'(\b_j)}{\zeta'(\b_j)}=\frac{\xi(s_{j+1})-\xi(s_j)}{\zeta(s_{j+1})-\zeta(s_j)}, \quad j=1,\ldots, k-1.
\end{equation}
We compute $\partial\varphi/\partial s_j$ for $j=2,\ldots,k-1$ (the cases $j=1$ and $j=k$ being similar). For such $j$, using relation \eqref{eq:b_j_def2}, rearranging the terms, and recalling \eqref{eq:xi_zeta}, we obtain
\begin{equation}
\Phi_j
=\frac{\zeta(s_j)[\xi(s_{j-1})-\xi(s_{j+1})]+\zeta(s_{j-1})[\xi(s_{j+1})-\xi(s_{j})]+\zeta(s_{j+1})[\xi(s_{j})-\xi(s_{j-1})]}{[\xi(s_j)\zeta(s_{j-1})-\xi(s_{j-1})\zeta(s_j)][\xi(s_{j+1})\zeta(s_j)-\xi(s_j)\zeta(s_{j+1}]}.
\end{equation}
When we compute $\partial\varphi/\partial s_j$, only the terms $\Phi_{j-1}$, $\Phi_j$ and $\Phi_{j+1}$ intervene. Some tedious computations provide
\begin{equation}
\frac{\partial \Phi_{j-1}}{\partial s_j} = \frac{\zeta'(s_j)[\xi(s_j)-\xi(s_{j-1})] - \xi'(s_j)[\zeta(s_j)-\zeta(s_{j-1})]}{[\xi(s_j)\zeta(s_{j-1})-\xi(s_{j-1})\zeta(s_j)]^2}
=: \frac{\zeta'(s_j)N_1 - \xi'(s_j)N_2}{D_1^2},
\end{equation}
\begin{equation}
\frac{\partial \Phi_{j+1}}{\partial s_j} = \frac{\zeta'(s_j)[\xi(s_{j+1})-\xi(s_{j})]-\xi'(s_j)[\zeta(s_{j+1})-\zeta(s_{j})]}{[\xi(s_{j+1})\zeta(s_{j})-\xi(s_{j})\zeta(s_{j+1})]^2}
=:\frac{\zeta'(s_j)N_3-\xi'(s_j)N_4}{D_2^2}
\end{equation}
and 
\begin{equation}
\frac{\partial \Phi_{j}}{\partial s_j} = \frac{\zeta'(s_j)N_5-\xi'(s_j)N_6}{D_1^2D_2^2},
\end{equation}
where
\begin{equation}
N_5:=\xi(s_{j-1})D_2[\zeta(s_{j+1})N_1-\xi(s_{j+1})N_2] -\xi(s_{j+1})D_1[\zeta(s_{j-1})N_3-\xi(s_{j-1})N_4],
\end{equation}
\begin{equation}
N_6:=\zeta(s_{j-1})D_2[\zeta(s_{j+1})N_1-\xi(s_{j+1})N_2] -\zeta(s_{j+1})D_1[\zeta(s_{j-1})N_3-\xi(s_{j-1})N_4].
\end{equation}
We sum the contributions to obtain, for $j=2,\ldots,k-1$,
\begin{multline}
\frac{\partial\varphi}{\partial s_j}=\frac{\partial \Phi_{j-1}}{\partial s_j}+\frac{\partial \Phi_{j}}{\partial s_j}+\frac{\partial \Phi_{j+1}}{\partial s_j} \\
= \frac{\Phi_j}{D_1 D_2} \left\{ \xi'(s_j)(N_2D_2+N_4D_1)-\zeta'(s_j)(N_1D_2+N_3D_1) \right\}.
\end{multline}
Therefore $\partial\varphi/\partial s_j=0$ if and only if \eqref{eq:reflection_law_continuousj} holds. Similarly, one can prove that $\partial\varphi/\partial s_1=0$ is equivalent to \eqref{eq:reflection_law_continuous1} and $\partial\varphi/\partial s_k=0$  is equivalent to \eqref{eq:reflection_law_continuousk}.
\end{proof}

We conclude this section with the following conjecture, which seems natural to us, since we have proved in Lemma \ref{lemma:uniqueness_1_layer} that the 1-layer solution of the limit problem is unique.

\begin{conjecture}
The configuration $(\b_1,\ldots,\b_{k-1})$ in Theorem \ref{thm:k-layer-limit} is unique.
\end{conjecture}

\section{Existence of the increasing and decreasing solutions}\label{S2}

\subsection{The increasing solution in the ball}

Let $0\leq\alpha<\beta\leq 1$. If $u\in H^1_{rad}(B_\b\setminus B_\a)$, we can assume it is continuous in $(\a,\b]$ ($\a$ included if positive) and the following set is well defined
\[
\cone_{+,[\a,\b]}=\{u\in H^1_{rad}(B_\b\setminus B_\a):\ u\geq0 \text{ and } u(r)\leq u(s) \text{ for every } \a < r\leq s \leq \b\}.
\]
Observe that if $\a=0$ and $u\in \cone_{+,[\a,\b]}$, then $u\in C(\overline{B_\b})$ and in particular it is a bounded function. In fact, since $u$ is non-decreasing, we can assume continuity also at the origin by letting $u(0)=\lim_{r\to0^+}u(r)$.
Moreover, $u$ is differentiable almost everywhere and $u'(r)\geq0$ where it is defined.

\begin{lemma}\label{lemma:properties_sol_cone}
Let $u\in\cone_{+,[\a,\b]}$ solve
\begin{equation}\label{eq:main_alpha_beta}
\left\{\begin{array}{ll}
-\Delta u+u=u^p & \text{ in } B_\b\setminus B_\a\\
\partial_\nu u=0 & \text{ on } \partial(B_\b\setminus B_\a).
\end{array}
\right.
\end{equation}
Then
\begin{itemize}
\item[(i)] either $u\equiv 1$, or $u(\a)<1$ and $u(\b)>1$;
\item[(ii)] $|u|\leq e^{1/2}$ in $\overline{B_\b\setminus B_\a}$;
\item[(iii)] $|u'|< 1$ in $\overline{B_\b\setminus B_\a}$.
\end{itemize}
\end{lemma}
\begin{proof}
(i) Integrating the equation for $u$ in $B_\b\setminus B_\a$ we obtain $\int_{B_\b\setminus B_\a} u(1-u^{p-1})\,dx=0$.

(ii) By multiplying the equation for $u$ by $u'$ we obtain
\begin{equation}\label{eq:u_u'}
u''u'-uu'+u^pu'=-\frac{N-1}{r}(u')^2.
\end{equation}
Hence the Lyapunov function
\[
L(r)=\frac{|u'(r)|^2}{2}-\frac{u(r)^2}{2}+\frac{u(r)^{p+1}}{p+1}
\]
satisfies $L'(r)=-\frac{N-1}{r}(u')^2\leq 0$. As a consequence, we have $L(r)\leq L(0)=-\frac{u(0)^2}{2}+\frac{u(0)^{p+1}}{p+1}\leq 0$ by point (i), for every $\a\leq r\leq \b$. This implies 
\begin{equation}\label{d0}
u(r)\leq u(\b) \leq \left(\frac{p+1}{2}\right)^{\frac{1}{p-1}}
\end{equation}
and hence the claim.

(iii) Since the function $\frac{x^2}{2}-\frac{x^{p+1}}{p+1}$ achieves its maximum at $x=1$, the inequality $L(r)\leq 0$ implies $|u'(r)|^2\leq \frac{p-1}{p+1}$.
\end{proof}


\begin{proposition}\label{prop:existence_increasing_sol}
Let $\lambda_2^{rad}(\a,\b)$ be the second radial eigenvalue of $-\Delta+Id$ in $B_\b\setminus B_\a$ with Neumann boundary conditions.
If $p>\lambda_2^{rad}(\a,\b)$ there exists $u_{p,+}(r)=u_{p,+}(r;\a,\b)\in\cone_{+,[\a,\b]}$ which solves \eqref{eq:main_alpha_beta} and such that a suitable rescaling achieves
\begin{equation}\label{eq:Q_definition}
c_{p,+}(\a,\b)=\inf \left\{Q_{p,[\a,\b]}(u): \, u\in\cone_{+,[\a,\b]}, \, \|u\|_\infty<\sqrt{e}+1 \right\},
\end{equation}
where $Q_{p,[\a,\b]}$ was defined in \eqref{eq:Q_p_def}.
By the maximum principle, $u_{p,+}$ is strictly increasing.
\end{proposition}
\begin{proof}
Fix $q>1$, $q<\frac{N+2}{N-2}$ if $N\geq3$, and $s_0=\sqrt{e}+1$. Define the following $C^1$ function
\[
f_p(s)=\left\{ \begin{array}{ll}
0 \ & \text{ if } s\leq 0 \\
s^p \ & \text{ if } 0\leq s\leq s_0 \\
s_0^p+p s_0^{p-1}(s-s_0) +(s-s_0)^q \ & \text{ if } s\geq s_0,
\end{array}\right.
\]
and let $F_p(s)=\int_0^s f(t)\,dt$,
\begin{equation}\label{eq:E_tilde}
\tilde{E}_{p,[\a,\b]}(u)=\int_{B_\b\setminus B_\a}\left( \frac{|\nabla u|^2}{2}+\frac{u^2}{2} -F_p(u) \right)\,dx.
\end{equation}

In \cite[Thm. 1.3, Prop. 4.7]{BNW} it is proved that there exists a strictly increasing radial solution of \eqref{eq:main_alpha_beta}, which achieves the following mountain pass level in $\cone_{+,[\a,\b]}$
\[
c'_{p,+}(\a,\b)=\inf_{\gamma\in\Gamma_{p,+}(\a,\b)} \max_{t\in[0,1]} \tilde{E}_{p,[\a,\b]}(\gamma(t)),
\]
where
\[
\Gamma_{p,+}(\a,\b)=\{\gamma\in C([0,1],\cone_{+,[\a,\b]}):\ \gamma(0)=0, \ \tilde{E}_{p,[\a,\b]}(\gamma(1))<0\}.
\]

Given this result, it will be enough to show that a suitable rescaling of this solution achieves $c_{p,+}(\a,\b)$. To this aim, let
\[
c''_{p,+}(\a,\b)=\inf\left\{\tilde{E}_{p,[\a,\b]}(u): \ u\in \cone_{+,[\a,\b]}, \, 
\|u\|_{H^1}^2=\int_{B_\b\setminus B_\a} f_p(u)u \,dx \right\}.
\]
It is standard to see that $c'_{p,+}(\a,\b)=c''_{p,+}(\a,\b)$, see for example \cite[Thm. 4.2]{WillemBook1996}, with the only difference that we have to work in the cone $\cone_{+,[\a,\b]}$. Therefore there exists $u_{p,+} \in\cone_{+,[\a,\b]}$, strictly increasing, which achieves  $c''_{p,[\a,\b]}$ and solves \eqref{eq:main_alpha_beta} with $f_p(u)$ in place of $u^p$. Since the conclusions of Lemma \ref{lemma:properties_sol_cone} still hold with $f_p(u)$ in place of $u^p$, we have
\[
|u_{p,+}| \leq e^{1/2} \text{ in } \overline{B_\b\setminus B_\a} \quad\text{and hence}\quad 
f_p(u_{p,+})=u_{p,+}^p \text{ in } \overline{B_\b\setminus B_\a}.
\]

Let us show that $w=c_{p,+}(\a,\b)^{-\frac{1}{p-1}} u_{p,+}$ achieves $c_{p,+}(\a,\b)$ and solves $-\Delta w + w=c_{p,+}(\a,\b) w^p$, which concludes the proof. On the one hand we have
\begin{equation}\label{eq:c''_p}
c_{p,+}(\a,\b)\leq Q_{p,[\a,\b]}(u_{p,+}) =\|u_{p,+}\|_{H^1}^{2\left(1-\frac{2}{p+1}\right)}
=\left( 2 \frac{p+1}{p-1} c''_{p,+}(\a,\b) \right)^\frac{p-1}{p+1}.
\end{equation}
On the other hand, $t_p w$ is an admissible test function for $c''_{p,+}(\a,\b)$, with
\[
t_p=\|w\|_{H^1}^\frac{2}{p-1} \|w\|_{p+1}^{-\frac{p+1}{p-1}}.
\]
Hence
\[
c''_{p,+}(\a,\b) \leq \tilde{E}_{p,[\a,\b]}(t_p w) =\frac{p-1}{2(p+1)} \|t_p w\|_{H^1}^2 
= \frac{p-1}{2(p+1)} c_{p,+}(\a,\b)^\frac{p+1}{p-1}.
\]
This implies that the inequalities in \eqref{eq:c''_p} are indeed equalities and in turn that $u_{p,+}$ can be chosen as a multiple of $w$.
\end{proof}

\begin{remark}\label{rem:existence_sol}
For a fix $p$, if $0<\bar{\a}<\bar{\b}$ are such that there exists the solution $u_{p,+}(\cdot;\bar{\a},\bar{\b})$, then by the continuity of $\lambda_2^{rad}(\a,\b)$, there exist $0<A_1<\bar{\a}<A_2$, $B_1<\bar{\b}<B_2$ such that the solution $u_{p,+}(\cdot;\a,\b)$ exists for every $(\a,\b)\in(A_1,A_2)\times(B_1,B_2)$. In case $\bar{\a}=0$, there exist $B_1<\bar{\b}<B_2$ such that the analogous holds in the ball.
\end{remark}

\begin{remark}\label{rem:tilde_Q_p}
We see from the previous proof that $u_{p,+}$ equivalently achieves
\[
\inf_{u\in\cone_{+,[\a,\b]}} \tilde{Q}_{p,[\a,\b]}(u),
\quad\text{ where } \tilde{Q}_{p,[\a,\b]}(u)=\frac{\|u\|^2_{H^1}}{\left(\int_{B_\b\setminus B_\a} F_p(u)\,dx\right)^{\frac{2}{p+1}}}.
\]
\end{remark}


As an additional information, we next show that the increasing solution is a local minimizer of $\tilde{Q}_{p,[\a,\b]}$ in $H^1_{rad}(B_\b\setminus B_\a)$. This implies for instance that the Morse index of the corresponding critical point of 
\[
E_{p,[\a,\b]}(u)=\int_{B_\b\setminus B_\a}\left( \frac{|\nabla u|^2}{2}+\frac{u^2}{2} -
\frac{u^{p+1}}{p+1} \right)\,dx.
\]
 is $1$. Indeed we have that $u_{p,+}$ is an eigenfunction of the operator 
$$v\mapsto -\Delta v + v - pu_{p,+}^{p-1}v$$
associated with the negative eigenvalue $1-p$ while for smooth functions $v$ orthogonal to $u_{p,+}$ in $H^1_{rad}(B_\b\setminus B_\a)$, we have 
$$E''_{p,[\a,\b]}(u_{p,+})[v,v]= \int_{B_\b\setminus B_\a}\left( {|\nabla v|^2}+v^2 -
pu_{p,+}^{p-1}v^{2} \right)\,dx\ge 0
$$
as a consequence of the fact that $u_{p,+}$ is a local minimizer of the functional $\tilde{Q}_{p,[\a,\b]}$. The claim then follows by density. 

\begin{theorem}\label{thm:morse_index}
The increasing solution $u_{p,+}$ is a local minimizer of $\tilde{Q}_{p,[\a,\b]}$ in $H^1_{rad}(B_\b\setminus B_\a)$.
\end{theorem}

We first show the minimality with respect to smooth variations. 

\begin{lemma}\label{lemma:C^2minimizer}
 There exists $\ep>0$ such that for every function satisfying
\begin{equation}\label{eq:phi_C2_close_u}
\varphi\in C_{rad}^2(B_\b\setminus B_\a), \ \varphi'(\a)=\varphi'(\b)=0, \ 
\|\varphi-u_{p,+}(\cdot;\a,\b)\|_{C^2}<\ep,
\end{equation}
it holds  $\tilde{Q}_{p,[\a,\b]}(u_{p,+})\leq \tilde{Q}_{p,[\a,\b]}(\varphi)$ ($\tilde{Q}_{p,[\a,\b]}$ is defined in Remark \ref{rem:tilde_Q_p}).
\end{lemma}
\begin{proof}
It will be enough to find $\ep>0$ such that $\varphi$ satisfying \eqref{eq:phi_C2_close_u} implies 
$\varphi\in\cone_{+,[\a,\b]}$. As $u_{p,+}(r)\geq u_{p,+}(\a)>0$, then for $\ep<u_{p,+}(\a)/2$ we have $\varphi>0$ in $B_\b\setminus B_\a$. Let us show that $\varphi$ is increasing.

Since $u_{p,+}'(\a)=u_{p,+}'(\b)=0$ and $u_{p,+}'(r)>0$ for $r\in(\a,\b)$, there exists $\bar{r}\in (\a,\b)$ such that
\[
\min_{r\in [\a,\a+\bar{r}]} u_{p,+}''(r)>0 \quad \text{and} \quad \max_{r\in [\b-\bar{r},\b]} u_{p,+}''(r)<0.
\]
By choosing
\[
\ep<\frac{1}{2}\min\left\{ \min_{[\a,\a+\bar{r}]} u_{p,+}'', - \max_{[\b-\bar{r},\b]} u_{p,+}'' \right\},
\]
we have $\varphi''>0$ in $[\a,\a+\bar{r}]$ and $\varphi''<0$ in $[\b-\bar{r},\b]$, for every $\varphi$ satisfying \eqref{eq:phi_C2_close_u}. Then, using the fact that $\varphi'(\a)=\varphi'(\b)=0$, we deduce
\[
\varphi'(r)=\int_\a^r \varphi''(s)\,ds>0, \ r\in (\a,\a+\bar{r}], \quad
\varphi'(r)=-\int_r^\b \varphi''(s)\,ds>0, \ r\in [\b-\bar{r},\b).
\]
Finally, since $u_{p,+}'(r)>0$ in $[\a+\bar{r},\b-\bar{r}]$, by choosing
\[
\ep<\frac{1}{2} \min_{r\in [\a+\bar{r},\b-\bar{r}]} u_{p,+}'(r),
\]
we also have $\varphi'>0$ in $(\a+\bar{r}, \b-\bar{r})$ for every $\varphi$ satisfying \eqref{eq:phi_C2_close_u}. Therefore $\varphi\in\cone_{+,[\a,\b]}$.
\end{proof}

\begin{proof}[Proof of Theorem \ref{thm:morse_index}]
Let us show that there exists $\ep>0$ such that
\begin{equation}\label{eq:phi_H1_close_u}
\varphi \in H^1_{rad}(B_\b\setminus B_\a), \ \|\varphi-u_{p,+}\|_{H^1}<\ep
\end{equation}
implies $\tilde{Q}_{p,[\a,\b]}(u_{p,+})\leq \tilde{Q}_{p,[\a,\b]}(\varphi)$.
We proceed as in \cite{BrezisNirenberg1993}. Suppose by contradiction that there exists a sequence $\varphi_n$ satisfying \eqref{eq:phi_H1_close_u} with $\ep=1/n$ and $\tilde{Q}_{p,[\a,\b]}(\varphi_n)<\tilde{Q}_{p,[\a,\b]}(u_{p,+})$. 
Since $\inf\{\tilde{Q}_{p,[\a,\b]}(\varphi):\ \varphi\in H^1_{rad}(B_\b\setminus B_\a), \ \|\varphi-u_{p,+}\|_{H^1}\leq 1/n\}$ is attained, we can assume that it is achieved by $\varphi_n$, so that
\begin{equation}\label{eq:brezis_nirenberg1}
\tilde{Q}_{p,[\a,\b]}'(\varphi_n)[\psi]= \mu_n (\varphi_n-u_{p,+},\psi)_{H^1}
\end{equation}
for some Lagrange multiplier $\mu_n$ and for very test function $\psi\in H^1(B_\b\setminus B_\a)$. Therefore $\varphi_n$ satisfies
\[
(1-\mu_n)(-\Delta\varphi_n+\varphi_n) =f_p(\varphi_n)-\mu_n(-\Delta u_{p,+}+u_{p,+}), \quad 
\varphi_n'(\a)=\varphi_n'(\b)=0.
\]

Let us show that $\mu_n<0$.
If $\|\varphi_n-u_{p,+}\|_{H^1}< 1/n$ then $\mu_n=0$. If otherwise $\|\varphi_n-u_{p,+}\|_{H^1}= 1/n$, let $t>0$ and $\psi$ be such that $\varphi_n+t\psi \in H^1_{rad}(B_\b\setminus B_\a)$ and $\|\varphi_n+t\psi -u_{p,+}\|_{H^1}\leq 1/n$. Then
\[
\frac{1}{n^2}\geq \| \varphi_n+t\psi-u_{p,+}\|_{H^1}^2=
\frac{1}{n^2}+2t(\varphi_n-u_{p,+},\psi)_{H^1} +t^2\|\psi\|^2_{H^1},
\]
so that
\begin{equation}\label{eq:brezis_nirenberg2}
2t(\varphi_n-u_{p,+},\psi)_{H^1}\leq 0.
\end{equation}
On the other hand, by the definition of $\varphi_n$, we have $\tilde{Q}_{p,[\a,\b]}(\varphi_n+t\psi)-\tilde{Q}_{p,[\a,\b]}(\varphi_n)\geq0$, which in the limit $t\to0$, $t>0$, provides $\tilde{Q}_{p,[\a,\b]}'(\varphi_n)[\psi]\geq0$. By comparing the last inequality with \eqref{eq:brezis_nirenberg1} and \eqref{eq:brezis_nirenberg2}, we obtain $\mu_n \leq 0$.

By using the equation satisfied by $u_{p,+}$, we can rewrite \eqref{eq:brezis_nirenberg1} as
\[
(1-\mu_n)\left\{ -\Delta(\varphi_n-u_{p,+})+\varphi_n-u_{p,+} \right\}=f_p(\varphi_n)-f_p(u_{p,+}).
\]
As $\varphi_n\to u_{p,+}$ in $H^1$ as $n\to\infty$ and $\mu_n\leq 0$, the bootrstap argument implies that $\varphi_n\to u_{p,+}$ in $C^2(B_\b\setminus B_\a)$. This contradicts Lemma \ref{lemma:C^2minimizer}, thus providing that $u_{p,+}$ locally minimizes $\tilde{Q}_{p,[\a,\b]}$ in the $H^1_{rad}$-topology.
\end{proof}


\subsection{The decreasing solution in the annulus}\label{S5}

As said before, finding a radial solution of 
\eqref{main-annulus} in an annulus
is easily done, whatever $p>1$, by minimizing the quotient $Q_{p,[\a,\b]}$ defined in \eqref{eq:Q_p_def} 
in $H^{1}_{rad}(B_\b\setminus B_\a)$. One expects that this produces a radially decreasing solution. One can show this fact for large $p$. In order to obtain a decreasing solution for a broader range of $p$ (we leave as a conjecture the fact that the minimizer of $Q_{p,[\a,\b]}$ is non increasing whatever $p>1$), we introduce the cone 
\[
\cone_{-,[\a,\b]}=\{u\in H^1_{rad}(B_\b\setminus B_\a):\ u\geq0 \text{ and } u(r)\geq u(s) \text{ for every } \a < r\leq s \leq \b\}.
\]
Observe that here we assume $\a>0$. Then $u\in C(\overline{B_\b\setminus B_\a})$ and in particular it is a bounded function. Moreover, $u$ is differentiable almost everywhere and $u'(r)\leq0$ where it is defined.

\begin{proposition}\label{prop:existence_decreasing_sol}
If $\alpha>0$ and $p>\lambda_2^{rad}(\a,\b)$ there exists $u_{p,-}(r)=u_{p,-}(r;\a,\b)\in\cone_{-,[\a,\b]}$ which solves \eqref{eq:main_alpha_beta} and such that a suitable rescaling achieves
\begin{equation}\label{eq:Q-_definition}
c_{p,-}(\a,\b)=\inf\{Q_{p,[\a,\b]}(u): \, u\in\cone_{-,[\a,\b]}\},
\end{equation}
\end{proposition}
The proof is classical. Observe that by the maximum principle, $u_{p,-}$ is strictly decreasing.

\section{Behaviour of the monotone solutions as $p\to+\infty$}\label{sec:behaviour_p_infty}\label{S3}

In this section we will prove the following convergence result.

\begin{proposition}\label{prop:convergence_green}
Denote by $u_{\infty,+}(r)=u_{\infty,+}(r;\a,\b)$ the unique solution of
\begin{equation}
\left\{\begin{array}{ll}\label{a1}
-\Delta u + u=0 \quad &\text{in } B_\b\setminus B_\a \\
\partial_\nu u=0 &\text{on } \partial B_\a \\
u=1 \quad &\text{on } \partial B_\b.
\end{array}\right.
\end{equation}
As $p\to\infty$ we have that $u_{p,+}\to u_{\infty,+}$ in $H^1(B_\b\setminus B_\a)\cap C^{0,\gamma}(\overline{B_\b\setminus B_\a})$ for every $\gamma\in(0,1)$.
\end{proposition}

The proof of this result is inspired by \cite{GN}. We divide it in several steps.

\begin{lemma}\label{lemma:weak_conv}
There exists $\bar{u} \in \cone_{+,[\a,\b]}$ satisfying $\|\bar{u}\|_{\infty}=\bar{u}(\b)=1$ such that, up to a subsequence, it holds
\[
u_{p,+} \rightharpoonup \bar{u} \text{ in } H^1(B_\b\setminus B_\a), \qquad 
u_{p,+}\to \bar{u} \text{ in } C^{0,\gamma}(\overline{B_\b\setminus B_\a}), \text{ for every } \gamma\in(0,1).
\]
\end{lemma}
\begin{proof}
Given any $\eta\in\cone_{+,[\a,\b]}$ with $\|\eta\|_\infty<\sqrt{e}+1$ we have, by the H\"older inequality (see \eqref{eq:Q_definition} for the definition of $c_{p,+}(\alpha,\beta)$),
\[
c_{p,+}(\a,\b)\leq \frac{\|\eta\|^2_{H^1}}{\|\eta\|^2_{p+1}}
\leq |B_\b\setminus B_\a|^\frac{p-1}{p+1} \frac{\|\eta\|^2_{H^1}}{\|\eta\|^2_{2}},
\]
which is bounded by a constant non depending on $p$. On the other hand, the equation for $u_{p,+}$ provides 
\begin{equation}\label{eq:H1_bound}
c_{p,+}(\a,\b)=\|u_{p,+}\|_{H^1}^{2\frac{p-1}{p+1}}= \|u_{p,+}\|_{p+1}^{p-1}.
\end{equation}
We deduce that the $H^1$-norm of the $u_{p,+}$ is bounded uniformly in $p$ and hence the weak convergence. 
The H\"older convergence comes from Lemma \ref{lemma:properties_sol_cone} (iii).

Being $u_{p,+}$ positive and strictly increasing for every $p$, $\bar{u}$ is non-negative and non-decreasing by the pointwise convergence. Let us show that $\|\bar{u}\|_{\infty}=1$. On the one hand, $\|\bar{u}\|_{\infty}\geq 1$ since $u_{p,+}(\b)>1$ for every $p$ and the convergence is $C^{0,\gamma}(\overline{B_\b\setminus B_\a})$. Suppose by contradiction that $\|\bar{u}\|_{\infty}=\bar{u}(\b)>1$. 
Then there exists $\bar{r}<\b$ and $\delta>0$ such that $u_{p,+}(r)>1+\delta$ for every $r\in(\bar{r},\b)$. By integrating \eqref{eq:main_alpha_beta} in $(\bar{r},\b)$ we obtain
\[
u'_{p,+}(\bar{r})>\frac{1}{\bar{r}^{N-1}} \int_{\bar{r}}^\b (u_{p,+}^{p-1}-1) r^{N-1}\,dr
\to\infty,
\]
thus contradicting Lemma \ref{lemma:properties_sol_cone} (iii).
\end{proof}

\begin{lemma}\label{lemma:limsup_u_p}
It holds
\[
\limsup_{p\to\infty}\|u_{p,+}\|_{p+1} \leq 1.
\]
\end{lemma}
\begin{proof}
For every $q>p$ we have, by the H\"older inequality,
\[
\|u_{p,+}\|_{p+1} \leq \|u_{p,+}\|_{q+1} |B_\b\setminus B_\a|^{\frac{1}{p+1}-\frac{1}{q+1}}.
\]
Therefore
\[
\begin{split}
\limsup_{p\to\infty}\|u_{p,+}\|_{p+1}  \leq 
\limsup_{p\to\infty} \lim_{q\to\infty} \left(\|u_{p,+}\|_{q+1} |B_\b\setminus B_\a|^{\frac{1}{p+1}-\frac{1}{q+1}}\right)\\
=  \limsup_{p\to\infty}\left(\|u_{p,+}\|_{\infty} |B_\b\setminus B_\a|^\frac{1}{p+1}\right) 
=\|\bar{u}\|_{\infty}=1,
\end{split}
\]
by Lemma \ref{lemma:weak_conv}.
\end{proof}

\begin{lemma}\label{lemma:w_p}
For every $u\in\cone_{+,[\a,\b]}$, $u\not\equiv0$, there exists a sequence $\{w_p\}\subset \cone_{+,[\a,\b]}$ such that
\[
\lim_{p\to\infty}\|w_p-u\|_{H^1}=0, \quad 
\|u\|_{\infty}\leq \liminf_{p\to\infty} \|w_p\|_{p+1}.
\]
\end{lemma}
\begin{proof}
We take $w_p$ of the form $\sigma_p u$ with $\sigma_p>1$, so that $w_p\in\cone_{+,[\a,\b]}$. In order to choose $\sigma_p$, let $f(\sigma)=\|\sigma u\|_{p+1}$.
Since $f$ is continuous, $f(1)<\|u\|_{\infty} |B_\b\setminus B_\a|^\frac{1}{p+1}$ and $f(\sigma)\to\infty$ as $\sigma\to\infty$, there exists $\sigma_p\in(1,\infty)$ such that $\|\sigma_p u\|_{p+1}=\|u\|_{\infty} |B_\b\setminus B_\a|^\frac{1}{p+1}$. It only remains to prove that $\sigma_p\to1$ as $p\to\infty$. Suppose on the contrary that $\sigma_p>1+\delta$ for some $\delta>0$ and for every $p$ large. Then
\[
\|u\|_{\infty}=\lim_{p\to\infty} \left( \|u\|_{\infty} |B_\b\setminus B_\a|^\frac{1}{p+1}\right)
=\lim_{p\to\infty} \|\sigma_p u\|_{p+1} 
>(1+\delta) \lim_{p\to\infty} \|u\|_{p+1},
\]
which provides $1>1+\delta$, a contradiction.
\end{proof}

\begin{proof}[Proof of Proposition \ref{prop:convergence_green}]
Let
\begin{equation}\label{eq:c_infty_def}
\begin{split}
c_{\infty,+} & =\inf \left\{\|u\|^2_{H^1}: \, u\in\cone_{+,[\a,\b]}, \, \|u\|_{\infty}=1 \right\} \\
& =\inf \left\{Q_{\infty,[\a,\b]}(u): \, u\in\cone_{+,[\a,\b]}, \, u\not\equiv0 \right\},
\end{split}
\end{equation}
where $Q_{\infty,[\a,\b]}$ is defined in \eqref{eq:Q_infty_def}.
By Lemma \ref{lemma:weak_conv} we have
\begin{equation}\label{eq:c_infty_less_c_p}
c_{\infty,+}\leq \|\bar{u}\|^2_{H^1} \leq \liminf_{p\to\infty}\|u_{p,+}\|^2_{H^1}
=\liminf_{p\to\infty}(c_{p,+} \|u_{p,+}\|^2_{p+1}).
\end{equation}
Using Lemma \ref{lemma:limsup_u_p} we conclude that $c_{\infty,+}\leq \liminf_{p\to\infty}c_{p,+}$. On the other hand, given any $u\in\cone_{+,[\a,\b]}$, $u\not\equiv0$, Lemma \ref{lemma:w_p} provides
\[
Q_{\infty,[\a,\b]}(u) \geq \limsup_{p\to\infty} Q_{p,[\a,\b]}(w_p) 
\geq \limsup_{p\to\infty} c_{p,+}.
\]
Therefore we have obtained
\[
c_{\infty,+}=\lim_{p\to\infty} c_{p,+}.
\]
In turn, the inequalities in \eqref{eq:c_infty_less_c_p} are indeed equalities, which implies both that $u_{p,+}\to \bar{u}$ in $H^1(B_\b\setminus B_\a)$ and that $\bar{u}$ achieves $c_{\infty,+}$ (with $\|\bar{u}\|_{\infty}=1$).

It only remains to show that $u_{\infty,+}$ is the unique function, having $L^\infty$-norm equal to 1, which achieves $c_{\infty,+}$. On the one hand, $u_{\infty,+}$ uniquely achieves
\[
\inf \left\{\|u\|^2_{H^1}: \, u=1 \text{ on } \partial B_\b \right\} \leq c_{\infty,+}.
\]
On the other hand, $u_{\infty,+}$ is radial and satisfies
\[
u_{\infty,+}'(r)=\frac{1}{r^{N-1}}\int_\a^r t^{N-1} u_{\infty,+}(t) \, dt \geq 0, \quad \forall \ r\in (\alpha,\beta),
\]
so that $u_{\infty,+}$ is an admissible test function for $c_{\infty,+}$.
\end{proof}
\begin{remark}
Note that we cannot have the $C^1$-convergence of the solution up to $r=\beta$. Indeed $u'_{p,+}(\beta)=0$ and $u'_{\infty,+}(\beta)>0$.
\end{remark}
An analogous result holds for the decreasing solution in the annulus.
\begin{proposition}\label{prop:convergence_green_decreasing}
Let $\a>0$.
Denote by $u_{\infty,-}(r)=u_{\infty,-}(r;\a,\b)$ the unique solution of
\begin{equation}
\left\{\begin{array}{ll}
-\Delta u + u=0 \quad &\text{in } B_\b\setminus B_\a \\
\partial_\nu u=0 &\text{on } \partial B_\b \\
u=1 \quad &\text{on } \partial B_\a.
\end{array}\right.
\end{equation}
As $p\to\infty$ we have that $u_{p,-}\to u_{\infty,-}$ in $H^1(B_\b\setminus B_\a)\cap C^{0,\gamma}(\overline{B_\b\setminus B_\a})$ for every $\gamma\in(0,1)$.
\end{proposition}

We conclude this section with a result that we will need later.

\begin{lemma}\label{lemma:u_p^p}
We have
\begin{equation}\label{0}
\lim_{p\to\infty} \frac{u_{p,+}(\b;\a,\b)^p}{p} =\frac12 \left( u_{\infty,+}'(\b;\a,\b) \right)^2 .
\end{equation}
\end{lemma}
\begin{proof}
The Pohozaev identity provides
\[
\begin{split}
\left(\frac{N-2}{2}-\frac{N}{p+1}\right) \int_{B_\b\setminus B_\a} |\nabla u_{p,+}|^2 \,dx
+\left(\frac{N}{2}-\frac{N}{p+1}\right) \int_{B_\b\setminus B_\a}u_{p,+}^2 \,dx \\
= \int_{\partial(B_\b\setminus B_\a)} \left(\frac{u_{p,+}^2}{2}
-\frac{u_{p,+}^{p+1}}{p+1}\right)\,d\sigma,
\end{split}
\]
so that
\begin{equation}\label{eq:pohozaev_u_p}
\begin{split}
\frac{|\partial B_\b|}{p+1} u_{p,+}(\b)^{p+1} 
=\frac{|\partial B_\a|}{p+1} u_{p,+}(\a)^{p+1} 
-\left(\frac{N-2}{2}-\frac{N}{p+1}\right) \int_{B_\b\setminus B_\a} |\nabla u_{p,+}|^2 \,dx \\
-\left(\frac{N}{2}-\frac{N}{p+1}\right) \int_{B_\b\setminus B_\a}u_{p,+}^2 \,dx
+\int_{\partial(B_\b\setminus B_\a)} \frac{u_{p,+}^2}{2} \,d\sigma.
\end{split}
\end{equation}
On the other hand, writing the Pohozaev identity satisfied by $u_{\infty,+}$ we obtain,
\begin{equation}\label{eq:pohozaev_u_infty}
\frac{|\partial B_\b|}2 u_{\infty,+}'(\b)^2
=  -\frac{N-2}{2} \int_{B_\b\setminus B_\a} |\nabla u_{\infty,+}|^2 \,dx 
-\frac{N}{2} \int_{B_\b\setminus B_\a}u_{\infty,+}^2\,dx
+\int_{\partial(B_\b\setminus B_\a)} \frac{u_\infty^2}{2} \,d\sigma.
\end{equation}
The convergence $u_{p,+}\to u_{\infty,+}$ in $H^1(B_\b\setminus B_\a)$ proved in Proposition \ref{prop:convergence_green} and the fact that $u_{p,+}(\a)<1$ imply that the right hand side in \eqref{eq:pohozaev_u_p} converges to the right hand side in \eqref{eq:pohozaev_u_infty}. 
\end{proof}


\section{Uniqueness and nondegeneracy of the monotone solutions}\label{S4}

\subsection{Uniqueness}
In this section we show that the minimal energy solution in the cone found in the previous section is unique.
\begin{theorem}\label{thm:uniqueness_minimal_energy_sol}
The value $c_{p,+}(\a,\b)$ is uniquely achieved by a multiple of $u_{p,+}(\cdot;\a,\b)$ for $p$ large enough.
\end{theorem}
\begin{proof}
\underline{Step 1}. Following \cite[Theorem 1.5]{G}, we perform a blow-up analysis of $u_{p,+}$.
Let
\begin{equation}\label{d3}
z_p(r)=\frac p{\|u_{p,+}\|_{\infty}} \left(u_{p,+}(\b+\e_p r)-\|u_{p,+}\|_{\infty}\right),
\qquad r \in \left[-\frac{\b-\a}{\e_p},0\right],
\end{equation}
where $\|u_{p,+}\|_\infty=u_{p,+}(\b)$ and 
\begin{equation}\label{eq:e_p_def}
p\e_p^2=\frac1{\|u_{p,+}\|_\infty^{p-1}}.
\end{equation}
From Lemma \ref{lemma:u_p^p} we obtain that, for $p$ large enough,
\begin{equation}\label{d5}
p\e_p=\frac{\sqrt p}{\|u_{p,+}\|_\infty^\frac{p-1}2}\to \frac{\sqrt{2}}{u'_{\infty,+}(\b)}
\quad\text{as } p\to\infty,
\end{equation}
so that, in particular, $\e_p\to 0$ as $p\to \infty$.

We claim that for every $R>0$ there exists $C>0$ independent of $p$ such that
\begin{equation}\label{d6}
|z_p(r)|+|z_p'(r)|\le C, \qquad r\in (-R,0).
\end{equation}
Of course, $z_p\leq0$. In order to obtain a bound from below, write
\[
z_p(-R)=-\frac{p\e_p R}{\|u_{p,+}\|_\infty} \cdot \frac{u_{p,+}(\b-\e_p R)-u_{p,+}(\b)}{-\e_p R }
=-C u_{p,+}'(\xi_p),
\]
for some $\xi_p\in (\b-\e_p R,\b)$, by the mean value theorem and \eqref{d5}.  The last quantity is bounded from below by Lemma \ref{lemma:properties_sol_cone} $iii)$.
This lemma, together with \eqref{d5}, also provides
\[
|z_p'|=\frac {p\e_p}{\|u_{p,+}\|_{\infty}}|u'_{p,+}| \leq C,
\]
so that \eqref{d6} is proved.

From \eqref{d6} and the equation solved by $z_p$:
\begin{equation}\label{d4}
\begin{cases}
-z_p''-\frac{(N-1)\e_p}{\b+\e_p r}z_p'+p\e_p^2 \left(1+\frac{z_p}p\right)
=\left(1+\frac{z_p}p\right)^p \quad &\text{ for } r\in\left(-\frac{\b-\a}{\e_p},0\right) \\
z_p(0)=z_p'(0)=0,
\end{cases}
\end{equation}
we can see that also $z''_p$ is bounded in $(-R,0)$. Therefore there exists $z_\infty\in C^1(-\infty,0)$ such that $z_p\to z_\infty$ in $C^1_{loc}(-\infty,0)$ and we can pass to the limit in \eqref{d4}, obtaining that $z_\infty$ satisfies
\begin{equation}\label{d7}
-z''=e^z\quad\hbox{in }(-\infty,0).
\end{equation}
All the solutions to this equation are given by
\begin{equation}\label{d8}
z(r)=\log\frac{4A^2e^{\sqrt2(A r+B)}}{\left(1+e^{\sqrt2(A r+B)}\right)^2},
\qquad A,B\in\R.
\end{equation}
Using that $z_\infty(0)=z_\infty'(0)=0$, we deduce
\begin{equation}\label{d9}
z_p(r)\to z_\infty(r)=\log\frac{4e^{\sqrt2r}}{\left(1+e^{\sqrt2r}\right)^2} \quad\text{in } C^1_{loc}(-\infty,0).
\end{equation}
\underline{Step 2}. 
We argue by contradiction and suppose that there exists $\tilde{u}_{p,+}(\cdot;\a,\b) \in \cone_{+,[\a,\b]}$, $\tilde{u}_{p,+}\not\equiv u_{p,+}$, which solves the equation and such that a suitable multiple achieves $c_{p,+}(\a,\b)$.
All the results proved in Sections \ref{sec:behaviour_p_infty}  apply to $\tilde{u}_{p,+}$ since it has the same variational characterization as $u_{p,+}$ and all the arguments can be repeated.

Since $\tilde{u}_{p,+}\not\equiv u_{p,+}$, the following normalized function is well defined
\begin{equation}\label{u2}
w_p=\frac{u_{p,+}-\tilde{u}_{p,+}}{\|u_{p,+}-\tilde{u}_{p,+}\|_\infty}.
\end{equation}
Letting $K_p(r)=\int_0^1(t u_{p,+}(r)+(1-t)\tilde{u}_{p,+}(r))^{p-1}dt$, we have that $w_p$ solves
\begin{equation}\label{u3}
\begin{cases}
-(r^{N-1}w_p')'+r^{N-1}w_p=r^{N-1}pK_pw_p \quad &\text{ for } r\in(\a,\b) \\
w_p'(\a)=w_p'(\b)=0, \quad |w_p|\leq1.
\end{cases}
\end{equation}
We claim that there exists $C>0$ independent of $p$ such that
\begin{equation}\label{eq:C1_bound_w_p}
|w_p'(r)| \leq C p \quad \text{for every } r\in \left(\frac{\a+\b}{2},\b\right).
\end{equation}
Integrating \eqref{u3} in $((\a+\b)/2,r)$, for $(\a+\b)/2\leq r\leq \b$, we obtain
\begin{equation}
|w_p'(r)|r^{N-1}\leq \int_{\frac{\a+\b}{2}}^r t^{N-1}|w_p(t)|dt
+p\int_{\frac{\a+\b}{2}}^rt^{N-1} |K_p(t)| |w_p(t)| dt.
\end{equation}
Since $|w_p|\le1$, we have
\begin{equation}\label{d14}
|w_p'(r)|\le C+p \int_{\frac{\a+\b}{2}}^r t^{N-1} |K_p(t)| dt .
\end{equation}
On the other hand we have that
\begin{equation}\label{eq:K_p_integrability}
\int_{\frac{\a+\b}{2}}^\b |K_p(r)|r^{N-1} dr \leq C.
\end{equation}
This comes from the inequality
\begin{equation}\label{u6}
|x^p-y^p|\leq p |x-y| \left(\max\{x,y\}\right)^{p-1}, \quad \text{for every } x,y>0,
\end{equation}
applied as follows
\begin{equation}
|K_p(r)|=\frac{|u_{p,+}(r)^p-\tilde{u}_{p,+}(r)^p|}{p|u_{p,+}(r)-\tilde{u}_{p,+}(r)|} 
\leq \left(\max\{u_{p,+}(r),\tilde{u}_{p,+}(r)\}\right)^{p-1},
\end{equation}
and from the fact that
\begin{equation}
\int_{B_\b\setminus B_\a} u_{p,+}^{p+1}\,dx 
+\int_{B_\b\setminus B_\a} \tilde{u}_{p,+}^{p+1} \,dx \leq C,
\end{equation}
uniformly in $p$, by the $H^1$-bound in Lemma \ref{lemma:weak_conv} and relation \eqref{eq:H1_bound}. So the claim \eqref{eq:C1_bound_w_p} is proved.

Thanks to this $C^1$-bound, we can perform a blow-up analysis of $w_p$, similar to the one in Step 1. Let $\e_p$ be as in \eqref{eq:e_p_def} and let $v_p(r)=w_p(\b+\e_p r)$ for $r\in\left(-(\b-\a)/\e_p,0\right)$, so that
\begin{equation}\label{u8}
\begin{cases}
-v_p''-\frac{(N-1)\e_p}{\b+\e_p r}v_p'+ \e_p^2 v_p = 
p\e_p^2 K_p(\b+\e_p r) v_p \quad &\text{ for } r\in\left(-\frac{\b-\a}{\e_p},0\right) \\
v_p'(\a)=v_p'(\b)=0 ,\ |v_p|\le1.
\end{cases}
\end{equation}
By \eqref{eq:C1_bound_w_p} we deduce that
\begin{equation}\label{u10}
|v_p'(r)|\le C \quad \text{for } r\in \left(-\frac{\b-\a}{2\e_p},0\right).
\end{equation}
Let us show that
\begin{equation}\label{u9}
p\e_p^2 K_p(\b+\e_p r)\rightarrow e^{z_\infty}
\end{equation}
locally in the compact subsets of $(-\infty,0)$, with $z_\infty$ defined in \eqref{d9}.
To this aim, consider the function $z_p$ introduced in \eqref{d3} and analogously let
\begin{equation}
p \tilde{\e}_p^2=\frac{1}{\|\tilde{u}_{p,+}\|_\infty^{p-1}}, \qquad
\tilde{z}_p(r)=\frac{p}{\|\tilde{u}_{p,+}\|_\infty} \left( \tilde{u}_{p,+}(\b+\tilde{\e}_p r)-\|\tilde{u}_{p,+}\|_\infty \right),
\end{equation}
so that
\begin{equation}
\tilde{u}_{p,+}(\b+\e_p r)
= \|\tilde{u}_{p,+}\|_\infty \left(1+\frac{1}{p}\tilde{z}_p\left(\frac{\e_p}{\tilde{\e}_p} r \right)\right).
\end{equation}
Since the asymptotic in \eqref{d5} holds for both $\e_p$ and $\tilde{\e}_p$, we deduce
\begin{equation}\label{u5}
\frac{\e_{p}}{\tilde{\e}_{p}}\to 1, \quad
\frac{\|u_{p,+}\|_\infty}{\|\tilde{u}_{p,+}\|_\infty} \to 1
\quad\hbox{as }p\rightarrow+\infty,
\end{equation} 
so that, using also \eqref{d9},
\begin{equation}
\tilde{z}_p\left(\frac{\e_p}{\tilde{\e}_p} r \right) \to z_\infty(r) \quad\text{as } p\to\infty.
\end{equation}
Using this and \eqref{u5} again, we have
\begin{eqnarray*}
 K_p(\b+\e_p r)&& =\| u_{p,+}\|_\infty^{p-1} \int_0^1\left\{ 1+\frac{1}{p}\left[ tz_p(r)+(1-t)\tilde{z}_p\left(\frac{\e_p}{\tilde{\e}_p}r\right) \right] +o_p(1) \right\}^{p-1}dt\nonumber \\
&& \sim \| u_{p,+}\|_\infty^{p-1} e^{z_\infty(r)} \quad\text{as } p\to\infty,
\end{eqnarray*}
proving \eqref{u9}.

By combining \eqref{u8}, \eqref{u10} and \eqref{u9}, we deduce that also $v_p''$ is bounded, so there exists $v_\infty\in C^1(-\infty,0)$  such that
\begin{equation}\label{d15}
v_p\to v_\infty \quad\hbox{in }C^1_{loc}(-\infty,0).
\end{equation}
Moreover  $v_\infty$ solves
\begin{equation}\label{d16}
\begin{cases}
-v''=e^{z_\infty} v \quad\hbox{in }(-\infty,0)\\
v'(0)=0,\ |v|\le1.
\end{cases}
\end{equation}
It is known, see \cite[Lemma 4.2]{G}, that
\begin{equation}\label{d17}
v(r)=A\frac{1-e^{\sqrt2r}}{1+e^{\sqrt2r}}+B\left(\sqrt2r\frac{1-e^{\sqrt2r}}{1+e^{\sqrt2r}}+2\right), \quad A,B\in\R.
\end{equation}
Since $v_\infty$ is bounded, we immediately obtain that $B=0$. On the other hand, the condition $v'_\infty(0)=0$ implies that $A=0$. Therefore
\begin{equation}\label{d18}
v_p\to v_\infty \equiv 0 \quad \text{in } C^1_{loc}(-\infty,0).
\end{equation}

\underline{Step 3}. We will see that $v_\infty\equiv0$ contradicts the fact that $\|w_p\|_\infty=1$ for every $p$.
Let $m_p\in [\a,\b]$ be such that $w_p(m_p)=1$ and let (up to a subsequence) $m_p\to m_\infty \in[\a,\b]$. 
Denoting by $G_{[\a,\b]}(r,t)$ the Green function of the operator $-u''-\frac{N-1}r u'+u$ with Neumann boundary conditions $u'(\a)=u'(\b)=0$, we have
\begin{equation}\label{48}
w_p(r)=\int_\a^\b G_{[\a,\b]}(r,t) p K_p(t)w_p(t) dt.
\end{equation}
Then
\[
1=w_p(m_p)=\int_\a^\b G_{[\a,\b]}(m_p,t) p K_p(t)w_p(t) dt 
= \int_{\frac{\a+\b}{2}}^\b G_{[\a,\b]}(m_p,t) p K_p(t)w_p(t) dt+o_p(1),
\]
where we used the fact that $u_{p,+}(t)\leq C<1$ for $t\in [(\a+\b)/2,\b]$. By applying the change of variables $t=\b+\e_p s$, we obtain
\begin{eqnarray}\label{49}
&&1= p\e_p\int_{-\frac{\b-\a}{2\e_p}}^0 G_{[\a,\b]}(m_p,\b+\e_ps)K_p(\b+\e_p s)v_p(s)ds +o_p(1)
\nonumber\\
&& = p\e_p G_{[\a,\b]}(m_p,\b)\int_{-\frac{\b-\a}{2\e_p}}^0 K_p(\b+\e_ps)v_p(s)ds \nonumber\\
&& +p\e_p \int_{-\frac{\b-\a}{2\e_p}}^0 \left(G_{[\a,\b]}(m_p,\b+\e_p s)-G_{[\a,\b]}(m_p,\b)\right)K_p(s)v_p(s)ds+o_p(1)
\nonumber\\
&&=: I_{1,p}+I_{2,p}+o_p(1).
\end{eqnarray}
Let us observe that there exists $C>0$ such that, for any $r,\alpha,\beta\in [0,1]$, the following holds
\begin{equation}\label{50}
|G(r,\a)-G(r,\beta)|\le C|\a-\beta|.
\end{equation}
This can be seen for example by making use of relations \eqref{eq:G}-\eqref{eq:xi_zeta_limits3}, we omit the details.
Then, using \eqref{u6} we can estimate $I_{2,p}$ as follows
\begin{eqnarray}\label{51}
&&\left|I_{2,p}\right| \le C p\e_p^2 \|u_{p,+}\|_\infty^p \int_{-\frac{\b-\a}{2\e_p}}^0 |s| 
\left(1+\frac{z_p(s)}p\right)^p |v_p(s)| ds
\nonumber\\
&&\leq C \|u_{p,+}\|_\infty \int_{-\frac{\b-\a}{2\e_p}}^0 |s| e^{z_p(s)} |v_p(s)| ds,
\end{eqnarray}
where we used \eqref{eq:e_p_def}.
We claim that
\begin{equation}\label{63}
\int_{-\frac{\b-\a}{2\e_p}}^0 |s| e^{z_p(s)} |v_p(s)| ds\to 0 \quad\text{as }p\to\infty.
\end{equation}
Since we know that $v_p\to 0$ a.e. as $p\rightarrow\infty$, it will be enough to show that the Lebesgue dominate convergence theorem applies. 
To this aim we use the following lemma, of which we postpone the proof to Section \ref{sec:proof_lemma_z_p}.

\begin{lemma}\label{lemma:z_p_derivative_large}
There exist $\bar{p}>1$, $\bar{s}\in (-(\b-\a)/(2\e_{\bar{p}}),0)$, $C>0$ such that
\begin{equation}\label{57}
z_p'(s)\ge C
\end{equation}
for every $p>\bar{p}$ and $s\in (-(\b-\a)/(2\e_{p}),\bar{s})$.
\end{lemma}

Given $\bar{p}$ and $\bar{s}$ as above, we compute
\begin{equation}\label{64}
z_p(s)=-\int_s^0 z_p'(\tau)d\tau =-\int_s^{\bar{s}} z_p'(\tau)d\tau-\int_{\bar{s}}^0 z_p'(\tau)d\tau \leq -C(\bar{s}-s) \leq C s,
\end{equation}
for every $p>\bar{p}$ and $s\in (-(\b-\a)/(2\e_{p}),\bar{s})$, which gives \eqref{63}. 

So we have proved that $|I_{2,p}|\rightarrow0$ as $p\rightarrow\infty$, so that \eqref{49} becomes
\begin{equation}\label{66}
1=p\e_p G_{[\a,\b]}(m_p,\b) \int_{-\frac{\b-\a}{2\e_p}}^0 K_p(\b+\e_ps)v_p(s)ds+o_p(1).
\end{equation}

On the other hand, \eqref{d18} implies (repeating the same procedure as before),
\begin{eqnarray}\label{67} 
&& o_p(1)=w_p(\b)=\int_\a^\b G_{[\a,\b]}(\b,t) p K_p(t)w_p(t)dt
=\int_{\frac{\a+\b}{2}}^\b G_{[\a,\b]}(\b,t) p K_p(t)w_p(t)dt+o_p(1)\nonumber\\
&&= p\e_p G_{[\a,\b]}(\b,\b) \int_{-\frac{\b-\a}{2\e_p}}^0 K_p(\b+\e_p s)v_p(s)ds
\nonumber\\
&& + p\e_p \int_{-\frac{\b-\a}{2\e_p}}^0 \left(G_{[\a,\b]}(\b,\b+\e_ps)-G_{[\a,\b]}(\b,\b)\right) 
K_p(\b+\e_p s)v_p(s)ds +o_p(1)
\nonumber\\
&&= p \e_p G_{[\a,\b]}(\b,\b) \int_{-\frac{\b-\a}{2\e_p}}^0 K_p(\b+\e_p s)v_p(s)ds+o_p(1),
\end{eqnarray}
which contradicts \eqref{66}. This concludes the proof of Theorem \ref{thm:uniqueness_minimal_energy_sol}.
\end{proof}

\begin{corollary}\label{coro:uniqueness_uniform}
Let us suppose that $(\a_n,\b_n)\rightarrow(\a,\b)$ with $\a<\b$. Then there exists $p_0=p_0(\a,\b)>1$ such that, for every $n$, the value $c_{p_0,+}(\a_n,\b_n)$ is uniquely achieved by a multiple of $u_{p,+}(\cdot;\a_n,\b_n)$.
\end{corollary}
\begin{proof}
It is enough to repeat step by step the proof of Theorem \ref{thm:uniqueness_minimal_energy_sol}. We just remark that the functions $z_p$ and $v_p$ are now defined in $\left(-\frac{\b_n-\a_n}{\e_p},0\right)$ and by assumption this interval converges again to $(-\infty,0)$. This applies also to Lemma \ref{lemma:z_p_derivative_large}.
\end{proof}

\subsection{Proof of Lemma \ref{lemma:z_p_derivative_large}}\label{sec:proof_lemma_z_p}

\begin{lemma}\label{52}
Recall that $c_{p,+}(\a,\b)$ is defined in Proposition \ref{prop:existence_increasing_sol}.
We have that 
\begin{equation}
c_{p,+}(\a,\b) = |\partial B_\b| u'_{\infty,+}(\b;\a,\b)+o_p(1) \quad \text{as }p\to\infty.
\end{equation}
\end{lemma}
\begin{proof}
By definition we have
\begin{equation}
c_{p,+}(\a,\b)=Q_{p,[\a,\b]}(u_{p,+})
= \left(|\partial B_1| \int_\a^\b u_{p,+}^{p+1}(t) t^{N-1}dt \right)^\frac{p-1}{p+1}.
\end{equation}
Then, recalling the blow up procedure \eqref{d3}, \eqref{d9}, we obtain
\begin{eqnarray}\label{53}
&& \frac{c_{p,+}(\a,\b)^\frac{p+1}{p-1}}{|\partial B_1|}
=\int_\a^\b u_{p,+}^{p+1}(t) t^{N-1} dt
=\e_p \int_{-\frac{\b-\a}{\e_p}}^0 u_{p,+}^{p+1}(\b+\e_p s)(\b+\e_ps )^{N-1} ds
\nonumber\\
&& = \e_p \|u_{p,+}\|_\infty^{p+1} \int_{-\frac{\b-\a}{\e_p}}^0 \left(1+\frac{z_p(s)}{p} \right)^{p+1}(\b+\e_p s)^{N-1} ds \nonumber\\
&&\ge \e_p \|u_{p,+}\|_\infty^{p+1} \b^{N-1} \int_{-\infty}^0 e^{z_\infty(s)}ds,
\end{eqnarray}
where in the last step we applied Fatou's lemma. Next by \eqref{eq:e_p_def}, \eqref{d9} and Lemma \ref{lemma:u_p^p}, we obtain
\[
\frac{c_{p,+}(\a,\b)^\frac{p+1}{p-1}}{|\partial B_1|}
\geq \b^{N-1} \left(\frac{\|u_{p,+}\|_\infty^p}{p}\right)^\frac12 (\sqrt2+o_p(1))
=\b^{N-1} u'_{\infty,+}(\b) +o_p(1).
\]
On the other hand, taking $u_{\infty,+}$ as test function for $c_{p,+}(\a,\b)$, we have
\begin{equation}\label{53b}
c_{p,+}(\a,\b) \le Q_{p,[\a,\b]}(u_{\infty,+})
=|\partial B_1| \b^{N-1} \frac{u'_{\infty,+}(\b)}{1+o_p(1)},
\end{equation}
where in the last line we used the equation satisfied by $u_{\infty,+}$.
\end{proof}

\begin{corollary}\label{54}
For every $\delta>0$ there exist $s(\delta)<0$ and $p(\delta)>1$ such that, for every $p>p(\delta)$ and $s\in(-(\b-\a)/\e_p,s(\delta))$, the following holds
\begin{equation}\label{55}
\int_{-\frac{\b-\a}{\e_p}}^s \left(1+\frac{z_p(\tau)}p\right)^{p+1} (\b+\e_p\tau)^{N-1}d\tau
<\delta.
\end{equation}
\end{corollary}
\begin{proof}
For any $\delta>0$ let us choose $s(\delta)$ such that
\begin{equation}
\b^{N-1}\int_{-\infty}^{s(\delta)}e^{z_\infty(\tau)}d\tau<\frac\delta3.
\end{equation}
We point out that a consequence of \eqref{53} and \eqref{53b} of Lemma \ref{52} is that
\begin{equation}
\int_{-\frac{\b-\a}{\e_p}}^0 \left(1+\frac{z_p(\tau)}p\right)^{p+1} (\b+\e_p\tau)^{N-1}d\tau\rightarrow \b^{N-1} \int_{-\infty}^0 e^{z_\infty(\tau)}d\tau,
\end{equation}
as $p\to\infty$, hence we can choose $p_1(\delta)$ such that, for every $p\geq p_1(\delta)$,
\begin{equation}
\left|\int_{-\frac{\b-\a}{\e_p}}^0 \left(1+\frac{z_p(\tau)}p\right)^{p+1} (\b+\e_p\tau)^{N-1}d\tau- \b^{N-1} \int_{-\infty}^0 e^{z_\infty(\tau)}d\tau\right|<\frac\delta3.
\end{equation}
Next, using the uniform convergence of $z_p$ to $z_\infty$ on the compact sets of $(-\infty,0]$ let us choose $p_2(\delta)$ such that, for $p\ge p_2(\delta)$,
\begin{equation}
\int_{s(\delta)}^0\left|\left(1+\frac{z_p(\tau)}p\right)^{p+1} (\b+\e_p\tau)^{N-1}-\b^{N-1}e^{z_\infty(\tau)}\right|d\tau<\frac\delta3.
\end{equation}
Finally, let $p_3(\delta)$ be such that $-(\b-\a)/\e_p< s(\delta)$ for every $p\geq p_3(\delta)$ and set $p(\delta):=\max\left\{p_1(\delta),p_2(\delta), p_3(\delta) \right\}$.

If $p>p(\delta)$ and $s\in(-(\b-\a)/\e_p,s(\delta))$, we have
\begin{equation}\label{u1}
\begin{split}
\int_{-\frac{\b-\a}{\e_p}}^s \left(1+\frac{z_p(\tau)}p\right)^{p+1} (\b+\e_p\tau)^{N-1}d\tau
\le \int_{-\frac{\b-\a}{\e_p}}^{s(\delta)} \left(1+\frac{z_p(\tau)}p\right)^{p+1} (\b+\e_p\tau)^{N-1}d\tau \\
= \b^{N-1}\int_{-\infty}^{s(\delta)}e^{z_\infty(\tau)}d\tau 
+ \int_{-\frac{\b-\a}{\e_p}}^0 \left(1+\frac{z_p(\tau)}p\right)^{p+1} (\b+\e_p\tau)^{N-1}d\tau
-\b^{N-1}\int_{-\infty}^0e^{z_\infty(\tau)}d\tau \\
- \left(\int_{s(\delta)}^0\left(1+\frac{z_p(\tau)}p\right)^{p+1} (\b+\e_p\tau)^{N-1}d\tau-\b^{N-1}\int_{s(\delta)}^0e^{z_\infty(\tau)}d\tau\right)<\delta,
\end{split}
\end{equation}
which proves the claim.
\end{proof}

\begin{proof}[Proof of Lemma \ref{lemma:z_p_derivative_large}]
Writing the equation satisfied by $u_{p,+}$ we obtain
\begin{equation}\label{58}
-u_{p,+}'(\b+\e_p s) (\b+\e_p s)^{N-1}+\int_\a^{\b+\e_p s}u_{p,+}(t) t^{N-1}
=\int_\a^{\b+\e_p s}u_{p,+}^p(t) t^{N-1} dt,
\end{equation}
which implies by \eqref{d3},
\begin{equation}
\begin{split}
\frac{\|u_{p,+}\|_\infty}{p\e_p} z_p'(s)(\b+\e_ps)^{N-1}
=\int_\a^{\b+\e_p s}u_{p,+}(t)t^{N-1}dt \\
- \e_p \|u_{p,+}\|_\infty^p \int_{-\frac{\b-\a}{\e_p}}^s\left(1+\frac{z_p(t)}{p}\right)^p (\b+\e_p t)^{N-1} dt.
\end{split}
\end{equation}
Using \eqref{eq:e_p_def}, \eqref{d5} and the fact that $s\in (-(\b-\a)/2\e_p,0)$, we obtain, for $p$ large enough,
\begin{equation}\label{60}
\begin{split}
z_p'(s) \ge \frac{1}{2} \frac{\sqrt2}{u'_{\infty,+}(\b)} \frac{1}{\b^{N-1}} 
\int_\a^{\frac{\a+\b}{2}} u_{p,+}(t)t^{N-1}dt \\
-2 \left(\frac{2}{\a+\b}\right)^{N-1} \int_{-\frac{\b-\a}{\e_p}}^s \left(1+\frac{z_p(t)}{p}\right)^p(\b+\e_p t)^{N-1}dt =:I_1-I_{2,p}(s).
\end{split}
\end{equation}
By Corollary \ref{54}, with $\e=I_1/2$, there exist $\bar{p}$ and $\bar{s}$ such that 
\begin{equation}
I_{2,p}(s)< \frac{I_1}{2} \quad \text{for every } p>\bar{p}, \ s\in \left(-\frac{\b-\a}{\e_p},\bar{s}\right).
\end{equation}
The H\"older inequality with exponents $(p+1)/p$ and $p+1$ provides
\begin{eqnarray}\label{61}
&&\int_{-\frac{\b-\a}{\e_p}}^s \left(1+\frac{z_p(t)}{p}\right)^p (\b+\e_p t)^{N-1} dt 
\nonumber\\
&& \le\left(\int_{-\frac{\b-\a}{\e_p}}^s \left(1+\frac{z_p(t)}{p}\right)^{p+1}(\b+\e_p t)^{N-1}dt\right)^\frac{p}{p+1} \cdot
\left(\int_{-\frac{\b-\a}{\e_p}}^0 (\b+\e_p t)^{N-1}dt\right)^\frac{1}{p+1} \nonumber\\
&& \le \left(\int_{-\frac{\b-\a}{\e_p}}^s \left(1+\frac{z_p(t)}{p}\right)^{p+1}(\b+\e_p t)^{N-1}dt\right)^\frac{p}{p+1} C \e_p^{-\frac{1}{p+1}}.
\end{eqnarray}
We notice that 
\begin{equation}\label{62}
\e_p^{-1}\sim \frac{p}{\sqrt{2}} u'_{\infty,+}(\b) \quad\text{so that}\quad \e_p^{-\frac{1}{p+1}} \to 1 \quad\text{as}\quad p\to\infty.
\end{equation}
By combining \eqref{60}-\eqref{62} we obtain that $z'_p(s)\geq I_1/2$ for every $p>\bar{p}$ and $s\in(-(\b-\a)/\e_p,\bar{s})$.
\end{proof}

\subsection{Nondegeneracy}
With very few changes with respect to the proof of the uniqueness, one can prove the following nondegeneracy result,
\begin{theorem}\label{thm:non_degeneracy}
Let $v_p$ solve
\begin{equation}\label{d1}
\begin{cases}
-(r^{N-1}v')'+r^{N-1}v=r^{N-1}pu_{p,+}^{p-1}v \quad &\text{ for } r\in(\a,\b) \\
v'(\a)=v'(\b)=0.
\end{cases}
\end{equation}
Then $v_p\equiv0$ for $p$ large.
\end{theorem}
\begin{proof}
As in the proof of Theorem \ref{thm:uniqueness_minimal_energy_sol}, Step 2, we suppose by contradiction that there exists a nontrivial solution of \eqref{d1}. The blow-up analysis of this solution can be performed exactly as in the proof to Theorem \ref{thm:uniqueness_minimal_energy_sol}, reaching the contradiction in the sae way. Here the calculations are indeed easier because there is only one blow-up parameter $\e_p$, hence \eqref{u5} holds automatically. Also, in the analogous of \eqref{u3} there is $u_p^p$ in place of $K_p$, so that \eqref{eq:K_p_integrability} is trivial.
\end{proof}

\subsection{$C^1$ dependence on the boundary points}
The following result is inspired from \cite[Lemma 3.4]{OrtegaVerzini2004}. 

\begin{lemma}\label{lemma:continuous_dependence}
Let $p$ be fixed and let $0<A_1<A_2<B_1<B_2$ be as in Remark \ref{rem:existence_sol}. Define
\[
I=\left\{ (r,\a,\b): \, A_1<\a<A_2, \, B_1<\b< B_2, \, \a<r<\b \right\}.
\]
Then the map $ I\ni (r,\a,\b) \mapsto u_{p,+}(r;\a,\b)$ is continuous. 

Similarly, in the case of the ball, let $0<B_1<B_2$ be as in Remark \ref{rem:existence_sol} and $I=\left\{ (r,\b): \, B_1<\b< B_2, \, 0\leq r<\b \right\}$. Then the map $ I\ni (r,\b) \mapsto u_{p,+}(r;0,\b)$ is continuous.
\end{lemma}
\begin{proof}
We prove the result in the case of the annulus, the case of the ball being analogous.
Let $(r,\a_n,\b_n)$ be a sequence in $I$ such that $\a_n\to\a_*$, $\b_n\to\b_*$. Let $\hat{u}_{p,+}(r;\a_n,\b_n)$ be the trivial extension of $u_{p,+}(r;\a_n,\b_n)$ in the interval $[A,B]:=[A_1,B_2]$ (extend as a constant outside $(\a_n,\b_n)$). Define $\hat{u}_{p,+}(r;\a_*,\b_*)$ analogously. Since $\{ \hat{u}_{p,+}(\cdot;\a_n,\b_n)  \}_n$ is bounded in $H^1(B_{B}\setminus B_{A})$, there exists $\tilde{u}\in H^1(B_{B}\setminus B_{A})$ such that (up to a subsequence)
\begin{equation*}
\hat{u}_{p,+}(\cdot;\a_n,\b_n) \rightharpoonup \tilde{u} \quad
\text{weakly in } H^1(B_{B}\setminus B_{A}).
\end{equation*}

In order to conclude the proof it will be enough to show that $\tilde{u}\equiv u_{p,+}(\cdot;\a_*,\b_*)$.
Let $\varphi\in C_c^\infty(B_{\b_*}\setminus B_{\a_*})$. Then $\varphi\in C_c^\infty(B_{\b_n}\setminus B_{\a_n})$ for $n$ sufficiently large and the $H^1$-weak convergence implies
\begin{equation*}
\int_{B_{\b_*}\setminus B_{\a_*}} \left( \nabla\tilde{u}\cdot\nabla\varphi +\tilde{u}\varphi\right) \,dx 
= \int_{B_{\b_*}\setminus B_{\a_*}} \tilde{u}^p\varphi \,dx.
\end{equation*}
Therefore both $\tilde{u}$ and $u_{p,+}(\cdot;\a_*,\b_*)$ solve equation \eqref{eq:main_alpha_beta} in $B_{\b_*}\setminus B_{\a_*}$. Moreover, by the pointwise convergence, $\tilde{u}$ is non-negative and non-decreasing (and hence positive and increasing by the maximum principle) and, by Lemma \ref{lemma:properties_sol_cone} (ii), it satisfies $\|\tilde{u}\|_\infty<\sqrt{e}+1$. Therefore $\tilde{u}$ can be used as a test function for $c_{p,+}(\a_*,\b_*)$.

Suppose by contradiction that $\tilde{u}\not\equiv u_{p,+}(\cdot;\a_*,\b_*)$. Then the uniqueness of the minimal energy solution in the cone with variable intervals proved in Corollary \ref{coro:uniqueness_uniform} implies
\begin{equation}\label{eq:continuous_dependence1}
Q_{p,[\a_*,\b_*]}(\tilde{u}) > Q_{p,[\a_*,\b_*]}(u_{p,+}(\cdot;\a_*,\b_*)).
\end{equation}
On the other hand, the $H^1$-weak convergence implies
\begin{equation}\label{eq:continuous_dependence2}
Q_{p,[\a_*,\b_*]}(\tilde{u}) 
\leq \liminf_{n\to\infty} Q_{p,[\a_*,\b_*]}(\hat{u}_{p,+}(\cdot;\a_n,\b_n)).
\end{equation}
We use \eqref{eq:continuous_dependence1}, \eqref{eq:continuous_dependence2} and the continuity if $Q_{p,[\a,\b]}$ with respect to $\a,\b$, to obtain
\begin{equation*}
\begin{split}
\lim_{n\to\infty} Q_{p,[\a_n,\b_n]}(\hat{u}_{p,+}(\cdot;\a_*,\b_*))
=Q_{p,[\a_*,\b_*]}(u_{p,+}(\cdot;\a_*,\b_*)) < Q_{p,[\a_*,\b_*]}(\tilde{u}) \\
\leq \liminf_{n\to\infty} Q_{p,[\a_*,\b_*]}(\hat{u}_{p,+}(\cdot;\a_n,\b_n))
=\lim_{n\to\infty} Q_{p,[\a_n,\b_n]}(u_{p,+}(\cdot;\a_n,\b_n)).
\end{split}
\end{equation*}
This implies that $\hat{u}_{p,+}(\cdot;\a_*,\b_*)$ achieves $c_{p,+}(\a_n,\b_n)$ for $n$ large, which contradicts Theorem \ref{thm:uniqueness_minimal_energy_sol}.\\
Finally, by the uniqueness of the minimal solution in $[\a,\b]$ it is standard to show that the sequences  $\a_n$ and $\b_n$ do converge.
\end{proof}

\begin{lemma}\label{lemma:C1_dependence}
In the same assumptions of the previous lemma, the maps $ I\ni (r,\a,\b) \mapsto u_{p,+}(r;\a,\b)$ and $ I\ni (r,\b) \mapsto u_{p,+}(r;0,\b)$ are of class $C^1$.
\end{lemma}
\begin{proof}
Again we prove the result only in the case of the annulus and we set $[A,B]:=[A_1,B_2]$. Let
\[
\hat{u}_{p,+}(s;\a,\b)=h^{\frac{2}{p-1}} u_{p,+}(hs+k;\a,\b),
\]
with
\[
h=\frac{\a-\b}{A-B} \quad\text{and}\quad k=\frac{A\b-B\a}{A-B}.
\]
Then $\hat{u}_{p,+} \in H^1_{rad}(B_B\setminus B_A)$ and solves
\[
-\hat{u}_{p,+}''-h\frac{N-1}{hs+k}\hat{u}_{p,+}'+ h^2 \hat{u}_{p,+}= (\hat{u}_{p,+})^p.
\]
We define a functional $\Phi:H^2_{rad}(B_B\setminus B_A)\cap\{u>0\} \times \left\{ (\a,\b): \, A_1<\a<A_2, \, B_1<\b< B_2 \right\} \to L^2(B_B\setminus B_A)$ as follows
\[
\Phi(u;\a,\b)=-u''-h\frac{N-1}{hs+k} u'+h^2u-u^p.
\]
The Implicit Function Theorem applies to $\Phi(u;\a,\b)=0$ near the point $(\hat{u}_{p,+};\a,\b)$. Indeed we have
\[
\partial_u \Phi(\hat{u}_{p,+}(\cdot;\a,\b);\a,\b))[\hat{\psi}] =
-\hat{\psi}''-h\frac{N-1}{hs+k}\hat{\psi}'+h^2\hat{\psi}-p(\hat{u}_{p,+})^{p-1}\hat{\psi}.
\]
Letting $\hat{\psi}(s)=\psi(hs+k)$ and rescaling back to the original variable $hs+k=r$, the previous expression becomes
\[
h^2(-\psi''-\frac{N-1}{r}\psi'+\psi-p u_{p,+}^{p-1}\psi).
\]
The nondegeneracy of $u_{p,+}$ proved in Theorem \ref{thm:non_degeneracy} implies that $\partial_u \Phi(\hat{u}_{p,+}(\cdot;\a,\b);\a,\b))$ is injective. Being a Fredholm operator of index 0, it is also surjective. 

By the Implicit Function Theorem there exists locally a $C^1$ map $(\a,\b)\mapsto u(\cdot;\a,\b)$ such that $\Phi(u(\cdot;\a,\b);\a,\b)=0$. Then Lemma \ref{lemma:continuous_dependence} implies that $(\a,\b)\mapsto u_{p,+}(\cdot;\a,\b)$ is of class $C^1$.
\end{proof}

\begin{lemma}\label{lemma:uniform_convergence}
Fix $0\leq\a<1$. For every $\e>0$ we have that
\begin{equation*}
u_{p,+}(\cdot;\a,\b)\to u_{\infty,+}(\cdot;\a,\b) \quad 
\text{in } H^1(B_\b\setminus B_\a) \cap C^{0,\gamma}(\overline{B_\b\setminus B_\a})
\text{ for every } \gamma\in(0,1),
\end{equation*}
as $p\to\infty$, uniformly in $\beta$ for $\a+\e\leq\b\leq 1$. Analogously, fix $0<\b\leq 1$, then for every $\e>0$ the convergence is uniform for $0\leq \a \leq \b-\e$.
\end{lemma}
\begin{proof}
We only prove the first statement.
First we claim that for every $\e>0$ there exists $C=C(\e)$ such that
\[
c_{p,+}(\a,\b)+\|u_{p,+}(\cdot;\a,\b)\|^2_{H^1}
+\|u_{p,+}(\cdot;\a,\b)\|^{p+1}_{p+1} \leq C 
\]
for every $p>1$, $\b\in [\a+\e,1]$. To prove the claim we proceed similarly to Lemma \ref{lemma:weak_conv}. Given any $\eta\in \cone_{+,[\a,1]}$ satisfying $\|\eta\|_\infty<\sqrt{e}+1$, we have
\begin{equation}\label{eq:uniform_weak_bound}
c_{p,+}(\a,\b)
\leq \frac{\|\eta\|^2_{H^1(B_\b\setminus B_\a)}}{\|\eta\|^2_{L^{p+1}(B_\b\setminus B_\a)}}
\leq |B_\b\setminus B_\a|^{\frac{p-1}{p+1}} 
\frac{\|\eta\|^2_{H^1(B_\b\setminus B_\a)}}{\|\eta\|^2_{L^2(B_\b\setminus B_\a)}}
\leq |B_1\setminus B_\a|^{\frac{p-1}{p+1}} 
\frac{\|\eta\|^2_{H^1(B_1\setminus B_\a)}}{\|\eta\|^2_{L^2(B_{\a+\e}\setminus B_\a)}}
\end{equation}
for every $p>1$, $\b\in [\a+\e,1]$. This together with \eqref{eq:H1_bound} proves the claim.

By Lemma \ref{lemma:properties_sol_cone} (iii) and Lemma \ref{lemma:continuous_dependence} we have
\[
|u'_{p,+}(\cdot;\a,\b)|<1, \qquad u_{p,+}(\cdot;\a,\b) \text{ is equicontinuous in } \b,
\]
which provides the uniform H\"older convergence by the Ascoli-Arzel\'a Theorem. Note that the equicontinuity in $\b$ of $u_{p,+}(\cdot;\a,\b)$ follows by Lemma \ref{lemma:partial_up_beta_bound} (which holds independently). \par

To prove the uniform $H^1$-convergence, we test the equation satisfied by $u_{p,+}(\cdot;\a,\b)-u_{\infty,+}(\cdot;\a,\b)$ by itself in $B_\b\setminus B_\a$ and apply the H\"older inequality to obtain
\[
\| u_{p,+}-u_{\infty,+}\|^2_{H^1(B_\b\setminus B_\a)} 
\leq \| u_{p,+}-u_{\infty,+}\|_{L^\infty(B_\b\setminus B_\a)}  
\|u_{p,+}\|^p_{L^p(B_\b\setminus B_\a)} 
+|\partial B_\b| |u'_{\infty,+}(\b)| |u_{p,+}(\beta)-1|.
\]
The uniform estimate \eqref{eq:uniform_weak_bound} and the uniform H\"older convergence allow to conclude.
\end{proof}

\begin{remark}\label{rem:uniform}
By combining the previous result with the proof of Lemma \eqref{0}, we see that
\begin{equation*}
\lim_{p\to\infty} \frac{u_{p,+}(\b;\a,\b)^p}{p} =\frac12 \left( u_{\infty,+}'(\b;\a,\b) \right)^2,
\end{equation*}
uniformly for $\a+\e\leq\b\leq 1$, for every $\e>0$.
\end{remark}



\section{Existence of the 1-layer solution}\label{S7}
For $k\in\N_0$ let $0= \b_0<\b_1<\ldots<\b_{k-1}<\b_k= 1$. In this section we prove the existence of a 1-layer radial solution of the equation \eqref{main-annulus} in the interval $[\b_{j-1},\b_j]$, for some $j=1,\ldots,k$, by gluing the increasing solution in $[\b_{j-1},\a]$ and the decreasing solution in $[\a,\b_{j}]$, for a suitable $\a\in (\b_{j-1},\b_j)$.

\begin{theorem}\label{thm:one_peak_solution}
For $p$ sufficiently large there exists a radial solution $u_{p,1\text{-layer}}(r;\b_{j-1},\b_j)$ of \eqref{main-annulus} in $B_{\b_j}\setminus B_{\b_{j-1}}$, having exactly one maximum point at $r=\alpha_{j,p}$. Furthermore, 
\begin{equation}
\alpha_{j,p}\to \alpha_{j}, \quad u_{p,1\text{-layer}} \to u_{\infty,1\text{-layer}} \ \text{ pointwise},
\end{equation}
as $p\to\infty$, where
\begin{equation}\label{eq:alpha1infty}
\left.\left(\frac{G_{[\b_{j-1},\b_{j}]}(r,r)}{r^{N-1}}\right)'\right|_{r=\a_{j}} =0,
\quad
u_{\infty,1\text{-layer}}(r;\b_{j-1},\b_{j})=\frac{G_{[\b_{j-1},\b_{j}]}(r,\a_{j})}{G_{[\b_{j-1},\b_{j}]}(\a_{j},\a_{j})}
\end{equation}
(compare with Definition \ref{def:1_layer_limit}).
\end{theorem}
\begin{proof}
We juxtapose the increasing solution $u_{p,+}(r;\beta_{j-1},\alpha)$ to the decreasing one $u_{p,-}(r;\alpha,\beta_{j})$. For a generic $\a \in (\beta_{j-1},\beta_j)$ this is a discontinuous function. Our aim is to find $\a_{j,p}$ such that it is continuous, that is to say
\begin{equation}\label{eq:nehari_neumann_match}
u_{p,+}(\a_{j,p};\b_{j-1},\a_{j,p})=u_{p,-}(\a_{j,p};\a_{j,p},\b_j).
\end{equation}
Since we are working with Neumann boundary conditions, the function
\begin{equation}\label{c3}
u_{p,1\text{-layer}}(r;\b_{j-1},\b_j)=\left\{\begin{array}{ll}
u_{p,+}(r;\b_{j-1},\a_{j,p}) & \text{ in } (\b_{j-1},\a_{j,p})\\
u_{p,-}(r;\a_{j,p},\b_j) & \text{ in } (\a_{j,p},\b_j)
\end{array}
\right.
\end{equation}
is the requested solution if $\a_{j,p}$ satisfies \eqref{eq:nehari_neumann_match}.

We define
\begin{equation}\label{c6}
L_p(\cdot;\b_{j-1},\b_j):\alpha\in(\b_{j-1},\b_j)
\mapsto \frac{u_{p,+}(\a;\b_{j-1},\a)^p-u_{p,-}(\a;\a,\b_j)^p}{p}.
\end{equation}
We aim to prove that $L_p$ has a zero for $p$ sufficiently large.
\begin{itemize}
\item[(i)] We proved in Lemma \ref{lemma:continuous_dependence} (and analogous result for $u_{p,-}$) that $L_p$ is continuous. This is a consequence of the uniqueness of the increasing and decreasing solutions.
\item[(ii)] Due to Lemma \ref{lemma:u_p^p} we have that
\begin{equation}\label{eq:L_infty}
L_p(\a;\b_{j-1},\b_j) \to L_\infty(\a;\b_{j-1},\b_j)=\frac{u'_{\infty,+}(\a;\b_{j-1},\a)^2-u'_{\infty,-}(\a;\a,\b_j)^2}{2},
\end{equation}
pointwise for $\a \in [\b_{j-1},\b_j]$. 
%
\item[(iii)]  By Propositions \ref{prop:convergence_green} and \ref{prop:convergence_green_decreasing} and by \eqref{eq:G_annulus} we have
\begin{equation}\label{C1}
u_{\infty,+}(r;\b_{j-1},\a)=\frac{G_{[\b_{j-1},\a]}(r,\a)}{G_{[\b_{j-1},\a]}(\a,\a)}
=\frac{\xi_{[\b_{j-1},\b_j]}(r)}{\xi_{[\b_{j-1},\b_j]}(\a)},
\end{equation}
\begin{equation}\label{C2}
u_{\infty,-}(r;\a,\b_j)=\frac{G_{[\a,\b_j]}(r,\a)}{G_{[\a,\b_j]}(\a,\a)}
=\frac{\zeta_{[\b_{j-1},\b_j]}(r)}{\zeta_{[\b_{j-1},\b_j]}(\a)}.
\end{equation}
Comparing with \eqref{eq:phi_derivative}, we obtain
\begin{equation*}
L_\infty(\a;\b_{j-1},\b_j)=\frac{1}{2} \left\{
\left(\frac{\xi'_{[\b_{j-1},\b_j]}(\a)}{\xi_{[\b_{j-1},\b_j]}(\a)}\right)^2
- \left(\frac{\zeta'_{[\b_{j-1},\b_j]}(\a)}{\zeta_{[\b_{j-1},\b_j]}(\a)}\right)^2 \right\}
=-\frac{\varphi'_{[\b_{j-1},\b_j]}(\a)}{2|\partial B_1|\a^{N-1}}.
\end{equation*}
Therefore Lemma \ref{lemma:uniqueness_1_layer} implies that
\begin{equation}\label{eq:L_infty_increasing}
\frac{\partial }{\partial\a} L_\infty(\a;\b_{j-1},\b_j) >0
\end{equation}
and that $L_\infty(\a;\b_{j-1},\b_j)$ admits a unique interior zero $\alpha_{j}$.
\end{itemize}
Combining (i)-(ii)-(ii) we deduce that $L_p$ has a zero for $p$ sufficiently large, which provides the existence of the 1-layer solution.
\end{proof}

\section{Existence of the $k$-layer solution}\label{S8}

In this section we write for shorter notation $u_{p,+}(r):=u_{p,+}(r;\a,b)$ and $u_{\infty,+}(r):=u_{\infty,+}(r;\a,b)$.

Let $p$ be fixed. Let us recall the definition of $I$ in Lemma \ref{lemma:continuous_dependence}.
In the case of the annulus we have
\[
I=\left\{ (r,\a,\b): \, A_1<\a<A_2, \, B_1<\b< B_2, \, \a<r<\b \right\},
\]
with $0<A_1<A_2<B_1<B_2$ as in Remark \ref{rem:existence_sol}.
In the case of the ball we have 
\[
I=\left\{ (r,\b): \, B_1<\b< B_2, \, 0\leq r<\b \right\},
\]
again with $0<B_1<B_2$ as in Remark \ref{rem:existence_sol}.

\begin{lemma}\label{lemma:partial_up_beta_bound}
Let $(r,\a,\b)\in I$. There exists $C>0$ independent of $\b$ and $p$ such that
\[
\left\|\frac{\partial u_{p,+}}{\partial\b}(\cdot;\a,\b)\right\|_\infty \leq C.
\]
\end{lemma}
\begin{proof}
Notice first that $\frac{\partial u_{p,+}}{\partial\b}(r):=\frac{\partial u_{p,+}}{\partial\b}(r;\a,\b)$ exists by Lemma \ref{lemma:C1_dependence} and solves
\begin{equation}\label{eq:partial_up_beta_equation}
\begin{cases}
-\left(\frac{\partial u_{p,+}}{\partial\b}\right)''-\frac{N-1}{r}\left(\frac{\partial u_{p,+}}{\partial\b}\right)'+\frac{\partial u_{p,+}}{\partial\b}=p u_{p,+}^{p-1}\frac{\partial u_{p,+}}{\partial\b}
 \quad\hbox{in }(\a,\b)\\
\left(\frac{\partial u_{p,+}}{\partial\b}\right)'(\b)=-u_{p,+}''(\b) \\
\left(\frac{\partial u_{p,+}}{\partial\b}\right)'(\a)=0.
\end{cases}
\end{equation}
\underline{Step 1}. Let $\ep_p$ be as in \eqref{eq:e_p_def} and $z_\infty$ be as in \eqref{d9}. We claim that there exists $A\neq0$ such that
\[
v_p(r):=\frac{\frac{\partial u_{p,+}}{\partial\b}(\b+\ep_pr)}{\left\|\frac{\partial u_{p,+}}{\partial\b}\right\|_\infty} \to A z_\infty'(r) \quad \text{in } C^1_{loc}(-\infty,0).
\]
To prove the claim we perform the same blow-up analysis as in the proof of Theorem \ref{thm:uniqueness_minimal_energy_sol}. We have (compare with \eqref{u8})
\begin{equation}
\begin{cases}
-v_p''-\frac{(N-1)\e_p}{\b+\e_p r}v_p'+ \e_p^2 v_p = 
\left(1+\frac{z_p}{p}\right)^{p-1} v_p 
\quad &\text{ for } r\in\left(-\frac{\b-\a}{\e_p},0\right) \\
|v_p|\le1,
\end{cases}
\end{equation}
where $z_p$ is as in \eqref{d3}. Proceeding as in \eqref{eq:C1_bound_w_p}-\eqref{d14} and \eqref{u10}, we can show that there exists $C>0$ independent of $p$ such that
\[
|v_p'(r)|\le C \quad \text{for } r\in \left(-\frac{\b-\a}{2\e_p},0\right).
\]
Indeed we have
\[
|v_p'(r)|\le C+\int_{-\frac{\b-\a}{2\ep_p}}^0 (\b+\ep_pt)^{N-1}\left(1+\frac{z_p}{p}\right)^{p-1} \,dt \leq C+\b^{N-1} \int_{-\infty}^0 e^{z_\infty}\,dt \leq C.
\]
Therefore there exists $v_\infty$ such that $v_p\to v_\infty$ in $C^1_{loc}(-\infty,0)$ and $v_\infty$ solves
\begin{equation}
\begin{cases}
-v'' =e^{z_\infty}  v
\quad &\text{ for } r\in\left(-\infty,0\right) \\
|v|\le1.
\end{cases}
\end{equation}
We deduce form \eqref{d17} that $v_\infty=A z'_\infty$. Finally, proceeding as in Step 3 of the proof of Theorem \ref{thm:uniqueness_minimal_energy_sol} one can show that $A\neq 0$.

\underline{Step 2}. Proceeding similarly to \eqref{eq:xi_xi'}-\eqref{eq:xi_xi'_tris}, we multiply \eqref{eq:main_alpha_beta} by $\frac{\partial u_{p,+}}{\partial\b} r^{N-1}$ and \eqref{eq:partial_up_beta_equation} by $u_{p,+}r^{N-1}$ and we integrate in $(\a,\b)$:
\begin{equation}\label{eq:partial_up_beta_derivative_up}
-\b^{N-1} \left(\frac{\partial u_{p,+}}{\partial\b} \right)'(\b) \, u_{p,+}(\b) 
= (p-1) \int_\a^\b u_{p,+}^p \frac{\partial u_{p,+}}{\partial\b} r^{N-1} \,dr.
\end{equation}
On the one hand, by Lemma \ref{lemma:u_p^p}, we have
\begin{equation}\label{eq:partial_up_beta_derivative}
\left(\frac{\partial u_{p,+}}{\partial\b}\right)'(\b)=-u_{p,+}''(\b)=u_{p;+}(\b)^p - u_{p,+}(\b)
=p\left(\frac{u_{\infty,+}'(\b)^2}{2}+o(1) \right).
\end{equation}
On the other hand, by performing the change of variables $r=\b+\ep_ps$, we obtain
\begin{multline}\label{eq:partial_up_beta_integral}
\int_\a^\b u_{p,+}^p \frac{\partial u_{p,+}}{\partial\b} r^{N-1} \,dr
= \ep_p \left\|\frac{\partial u_{p,+}}{\partial\b}\right\|_{\infty} \|u_{p,+}\|_{\infty}^p 
\int_{-\frac{\b-\a}{\ep_p}}^0 \frac{u_{p,+}(\b+\ep_ps)^p}{\|u_{p,+}\|_\infty^p} v_p(s) (\b+\ep_ps)^{N-1} \,ds,
\end{multline}
with $v_p$ as in Step 1. Recalling that (see \eqref{d3}, \eqref{d5} and \eqref{d9})
\[
\begin{split}
\e_p\|u_{p,+}\|_\infty^p=\frac{\e_p\|u_{p,+}\|_\infty^{\frac{p+1}{2}}}{\sqrt{p}} \to \frac{u_{\infty,+}'(\b)}{\sqrt{2}}, \\
\frac{u_{p,+}(\b+\ep_ps)^p}{\|u_{p,+}\|_\infty^p} \to e^{z_\infty},
\end{split}
\]
and using \eqref{64} to pass to the limit in \eqref{eq:partial_up_beta_integral}, we get
\begin{multline}\label{eq:partial_up_beta_limit}
\int_\a^\b u_{p,+}^p \frac{\partial u_{p,+}}{\partial\b} r^{N-1} \,dr
= \left\|\frac{\partial u_{p,+}}{\partial\b}\right\|_{\infty} \left( \frac{u_{\infty,+}'(\b)}{\sqrt{2}}+o(1) \right) \b^{N-1} A 
\int_{-\infty}^0 (e^{z_\infty} z_\infty'+o(1)) \,ds.
\end{multline}
Since $A\neq0$ and $u_{\infty,+}'(\b)\neq0$, by combining \eqref{eq:partial_up_beta_derivative_up}, \eqref{eq:partial_up_beta_derivative} and \eqref{eq:partial_up_beta_limit} we deduce that $\left\|\frac{\partial u_{p,+}}{\partial\b}\right\|_{\infty}$ is bounded.
\end{proof}

\begin{corollary}\label{coro:partial_up_beta_convergence}
Let $(r,\a,\b)\in I$. There exists a function $C(\b)$ such that
\[
\frac{\partial u_{p,+}}{\partial\b} \to C(\b) u_{\infty,+}
\]
pointwise as $p\to\infty$.
\end{corollary}

\begin{lemma}\label{lemma:u_infty_integrals}
We have
\begin{equation}\label{eq:u_infty_equality1}
(N-1)\int_\a^\b u_{\infty,+}' u_{\infty,+} r^{N-3} \,dr=\b^{N-1} \left( u_{\infty,+}''(\b)-u_{\infty,+}'(\b)^2 \right) - \a^{N-1} u_{\infty,+}(\a)^2,
\end{equation}
\begin{equation}\label{eq:u_infty_equality2}
2\int_\a^\b u_{\infty,+}^2 r^{N-1} \,dr = \b^{N-1} \left( u_{\infty,+}'(\b)+\b u_{\infty,+}''(\b) \right) -\b^N u_{\infty,+}'(\b)^2 -\a^{N-1} u_{\infty,+}(\a)^2.
\end{equation}
\end{lemma}
\begin{proof}
Proceeding similarly to \eqref{eq:xi_xi'}-\eqref{eq:xi_xi'_tris}, we multiply \eqref{a1} by $r^{N-1} u_{\infty,+}'$ and we multiply the equation satisfied by $u_{\infty,+}'$ by $r^{N-1} u_{\infty,+}$ to obtain
\[
(r^{N-1}u_{\infty,+}''u_{\infty,+})'-(r^{N-1}(u_{\infty,+}')^2)'=(N-1) u_{\infty,+}'u_{\infty,+}r^{N-3}.
\]
Intergating in $(\a,\b)$ and recalling that $u_{\infty,+}'(\a)=0$, $u_{\infty,+}''(\a)=u_{\infty,+}(\a)$, $u_{\infty,+}(\b)=1$, we obtain \eqref{eq:u_infty_equality1}.

Let $k=ru_{\infty,+}'$ so that
\[
-k''-\frac{N-1}{r} k' +k =-2u_{\infty,+}.
\]
We multiply the last equation by $r^{N-1}u_{\infty,+}$ and \eqref{a1} by $r^{N-1}k$ to obtain
\[
(k' u_{\infty,+}r^{N-1})' -((u_{\infty,+})^2 r^N)' = 2 u_{\infty,+}^2 r^{N-1}.
\]
Intagrating in $(\a,\b)$ and noticing that $k'(\b)=u_{\infty,+}'(\b)+\b u_{\infty,+}''(\b)$ and that $k'(\a)=\a u_{\infty,+}(\a)$, we obtain \eqref{eq:u_infty_equality2}.
\end{proof}

\begin{lemma}\label{lemma:partial_up_beta_convergence}
For $(r,\a,\b)\in I$ we have
\[
\left. p \frac{\partial u_{p,+}}{\partial\b}(r;\a,\b)\right|_{r=\b} = 2 \frac{u_{\infty,+}''(\b)-(u_{\infty,+}'(\b))^2}{u_{\infty,+}'(\b)} +o(1).
\]
\end{lemma}
\begin{proof}
We define $w:=u_{p,+}'$, so that
\begin{equation}\label{eq:w}
\begin{cases}
-w''-\frac{N-1}{r}w'+w=p u_{p,+}^{p-1}w -\frac{N-1}{r^2}w
\quad &\text{ for } r\in (\a,\b) \\
w(\a)=w(\b)=0 \\
w'(\a)=u_{p,+}''(\a), \ w'(\b)=u_{p,+}''(\b),
\end{cases}
\end{equation}
and $z:=ru_{p,+}'+\frac{2}{p-1}u_{p,+}$, so that
\begin{equation}\label{eq:z}
\begin{cases}
-z''-\frac{N-1}{r}z'+z=p u_{p,+}^{p-1}z - 2 u_{p,+}
\quad &\text{ for } r\in (\a,\b) \\
z'(\b)=\b u_{p,+}''(\b), \  z'(\a)=\a u_{p,+}''(\a) \\
z(\b)=\frac{2}{p-1}u_{p,+}(\b), \ z(\a)=\frac{2}{p-1}u_{p,+}(\a).
\end{cases}
\end{equation}
We multiply \eqref{eq:w} by $r^{N-1}\frac{\partial u_{p,+}}{\partial\b}$ and \eqref{eq:partial_up_beta_equation} by $r^{N-1}w$ and integrate in $(\a,\b)$:
\begin{equation}\label{eq:partial_up_beta_w}
\b^{N-1} u_{p,+}''(\b) \frac{\partial u_{p,+}}{\partial\b}(\b)
-\a^{N-1} u_{p,+}''(\a) \frac{\partial u_{p,+}}{\partial\b}(\a)
=(N-1) \int_\a^\b u_{p,+}' \frac{\partial u_{p,+}}{\partial\b} r^{N-3} \,dr.
\end{equation}
Similarly, we multiply \eqref{eq:z} by $r^{N-1}\frac{\partial u_{p,+}}{\partial\b}$ and \eqref{eq:partial_up_beta_equation} by $r^{N-1}z$ and we integrate in $(\a,\b)$:
\begin{multline}\label{eq:partial_up_beta_z}
\b^{N} u_{p,+}''(\b) \frac{\partial u_{p,+}}{\partial\b}(\b)
-\a^{N} u_{p,+}''(\a) \frac{\partial u_{p,+}}{\partial\b}(\a)
+\frac{2}{p-1} \b^{N-1} u_{p,+}''(\b) u_{p,+}(\b) \\
=2\int_\a^\b u_{p,+} \frac{\partial u_{p,+}}{\partial\b} r^{N-1}\,dr.
\end{multline}
Using Lemma \ref{lemma:uniform_convergence}, equation \eqref{eq:partial_up_beta_derivative}, Corollary \ref{coro:partial_up_beta_convergence} and Lemma \ref{lemma:u_infty_integrals}, we can pass to the limit in \eqref{eq:partial_up_beta_w} and \eqref{eq:partial_up_beta_z} to obtain
\begin{multline}\label{eq:C_beta1}
-\b^{N-1} p \left(\frac{u_{\infty,+}'(\b)^2}{2}+o(1) \right) \frac{\partial u_{p,+}}{\partial\b}(\b)
-C(\b)\a^{N-1}u_{\infty,+}(\a)^2  \\
=C(\b)\left\{ \b^{N-1} \left( u_{\infty,+}''(\b)-u_{\infty,+}'(\b)^2 \right) - \a^{N-1} u_{\infty,+}(\a)^2 \right\}+o(1)
\end{multline}
and
\begin{multline}\label{eq:C_beta2}
-\b^{N} p \left(\frac{u_{\infty,+}'(\b)^2}{2}+o(1) \right) \frac{\partial u_{p,+}}{\partial\b}(\b)
-C(\b)\a^{N}u_{\infty,+}(\a)^2 -\frac{2p}{p-1} \b^{N-1} \left(\frac{u_{\infty,+}'(\b)^2}{2}+o(1) \right) \\
=C(\b) \left\{\b^{N-1} \left( u_{\infty,+}'(\b)+\b u_{\infty,+}''(\b) \right) -\b^N u_{\infty,+}'(\b)^2 -\a^{N-1} u_{\infty,+}(\a)^2 \right\} +o(1).
\end{multline}
By combining the two previous expressions we obtain the statement.
\end{proof}

\begin{remark}\label{rem:C_beta}
We infer from \eqref{eq:C_beta1} and \eqref{eq:C_beta2} that $C(\beta)=- u_{\infty,+}'(\beta)$, so that Corollary \ref{coro:partial_up_beta_convergence} provides
\[
\frac{\partial u_{p,+}}{\partial\b} \to - u_{\infty,+}'(\beta) u_{\infty,+}
\]
pointwise as $p\to\infty$.
\end{remark}

\begin{lemma}\label{lemma:uniform_C1_conv}
The convergence in Lemma \ref{lemma:partial_up_beta_convergence} is uniform in $\b$ for $(r,\a,\b)\in I$.
\end{lemma}
\begin{proof}
We argue by contradiction and suppose that
\begin{equation}
\lim_{p\to+\infty} \sup_{\b\in[B_1,B_2]} \left[ \left. p \frac{\partial u_{p,+}}{\partial\b}(r;\a,\b)\right|_{r=\b} - 2 \frac{u_{\infty,+}''(\b)-(u_{\infty,+}'(\b))^2}{u_{\infty,+}'(\b)} \right] \geq C>0.
\end{equation}
So we can select sequences $p_n\to+\infty$ and $\b_n\rightarrow\b_0\in[B_1,B_2]$ such that
\begin{equation}
 \left. p_n \frac{\partial u_{p_n,+}}{\partial\b}(r;\a,\b_n)\right|_{r=\b_n} - 2 \frac{u_{\infty,+}''(\b_n)-(u_{\infty,+}'(\b_n))^2}{u_{\infty,+}'(\b_n)} \to C>0,
\end{equation}
and the smoothness of $u_{\infty,+}$ implies
\begin{equation}\label{eq:uniform_contradiction}
 \left. p_n \frac{\partial u_{p_n,+}}{\partial\b}(r;\a,\b_n)\right|_{r=\b_n} - 2 \frac{u_{\infty,+}''(\b_0)-(u_{\infty,+}'(\b_0))^2}{u_{\infty,+}'(\b_0)} \to C>0.
\end{equation}
Let us consider relation \eqref{eq:partial_up_beta_w} evaluated along the sequences $p_n$, $\b_n$. By Remark \ref{rem:uniform} we have
\[
-\frac{u_{p_n,+}''(\b_n)}{p_n} 
\to  \frac{u_{\infty,+}'(\b_0)^2}{2} .
\]
uniformly. Moreover by repeating the proof of Lemma \ref{lemma:partial_up_beta_bound} with $p=p_n$ and $\b=\b_n$ one can prove that
\[
\left\|\frac{\partial u_{p_n,+}}{\partial\b}(\cdot;\a,\b_n)\right\|_\infty \leq C,
\quad \text{with }C \text{ independent of } n.
\]
Lemma \ref{lemma:uniform_convergence} and Remark \ref{rem:C_beta} provide
\begin{equation}
u_{p_n,+}'(r;\a,\b_n) \to u_{\infty,+}'(r;\a,\b_0), \quad
\frac{\partial u_{p_n,+}}{\partial\b}(r;\a,\b_n) \to - u_{\infty,+}'(\beta_0) u_{\infty,+}(r)
\end{equation}
pointwise. We can apply the dominated convergence theorem to pass to the limit in  \eqref{eq:partial_up_beta_w} to obtain
\begin{equation}
p_n \frac{\partial u_{p_n,+}}{\partial\b}(r;\a,\b_n)- 2 \frac{u_{\infty,+}''(\b_0)-(u_{\infty,+}'(\b_0))^2}{u_{\infty,+}'(\b_0)} \to 0,
\end{equation}
which contradicts \eqref{eq:uniform_contradiction}.
\end{proof}

We have an anologous result for the decreasing solution.

\begin{lemma}\label{lemma:partial_up-_beta_convergence}
We have
\[
\left. p \frac{\partial u_{p,-}}{\partial\a}(r;\a,\b)\right|_{r=\a} \to
 2 \frac{u_{\infty,-}''(\a)-(u_{\infty,-}'(\a))^2}{u_{\infty,-}'(\a)}
\]
uniformly in $\b$ for $(r,\a,\b)\in I$.
\end{lemma}

\begin{theorem}\label{thm:C1convergence}
Let $L_p(\a;\b_{j-1},\b_j)$ and $L_\infty(\a;\b_{j-1},\b_j)$ be defined in \eqref{c6} and \eqref{eq:L_infty} respectively. There exists $\ep>0$ such that
\begin{equation}
L_p(\cdot;\b_{j-1},\b_j)\to L_\infty(\cdot;\b_{j-1},\b_j) \quad\text{in } C^1(\a_j-\ep,\a_j+\ep).
\end{equation}
\end{theorem}
\begin{proof}
On the one hand we have
\[
\frac{\partial}{\partial\a} \left( u_{p,+}(\a;\b_{j-1},\a)\right)
=u_{p,+}'(\a;\b_{j-1},\a)+ \left. \frac{\partial u_{p,+}}{\partial\a}(r;\b_{j-1},\a)\right|_{r=\a}
=\left. \frac{\partial u_{p,+}}{\partial\a}(r;\b_{j-1},\a)\right|_{r=\a}
\]
so that, by Lemmas \ref{lemma:u_p^p} and \ref{lemma:partial_up_beta_convergence},
\begin{multline}
\frac{\partial}{\partial\a} \left( \frac{u_{p,+}(\a;\b_{j-1},\a)^p}{p} \right)
= \frac{u_{p,+}(\a;\b_{j-1},\a)^{p-1}}{p} p \left. \frac{\partial u_{p,+}}{\partial\a}(r;\b_{j-1},\a)\right|_{r=\a} \\
= \frac12 \left( u_{\infty,+}'(\a;\b_{j-1},\a) \right)^2 2 \frac{u_{\infty,+}''(\a;\b_{j-1},\a)-(u_{\infty,+}'(\a;\b_{j-1},\a))^2}{u_{\infty,+}'(\a;\b_{j-1},\a)} +o(1).
\end{multline}
On the other hand, by computing explicitely the derivatives in \eqref{C1}, we obtain
\begin{equation}
\left.\frac{\partial u_{\infty,+}'}{\partial \a} (r;\b_{j-1},\a) \right|_{r=\a}=-(u_{\infty,+}'(\a;\b_{j-1},\a))^2,
\end{equation}
and hence
\begin{equation}
\frac{\partial}{\partial\a} \left( \frac{(u_{\infty,+}'(\a;\b_{j-1},\a))^2}{2} \right)
=u_{\infty,+}'(\a;\b_{j-1},\a)[u_{\infty,+}''(\a;\b_{j-1},\a)-(u_{\infty,+}'(\a;\b_{j-1},\a))^2].
\end{equation}
The convergence is uniform by Lemma \ref{lemma:uniform_C1_conv} and by Remark Remark \ref{rem:uniform}. Since an analogous result hold for the decreasing solution, the statement is proved.
\end{proof}

\begin{corollary}\label{coro:alpha_jp_C1}
The map $\a_{j,p}(\b_{j-1},\b_j)$ defined in Theorem \ref{thm:one_peak_solution} is of class $C^1$.
\end{corollary}
\begin{proof}
$\a_{j,p}$ is implicitely defined by the equation $L_p(\a_{j,p};\b_{j-1},\b_j)=0$, with $L_p$ as in \eqref{c6}. We infer from Theorem \ref{thm:C1convergence} and relation \eqref{eq:L_infty_increasing} that 
\[
\frac{\partial }{\partial\a} L_p(\a;\b_{j-1},\b_j) >0,
\]
so that the Implicit Function Theorem applies.
\end{proof}

\begin{corollary}\label{coro:1_layer_C1}
$u_{p,1\text{-layer}}(\b_j;\b_j,\b_{j+1})$ is $C^1$ in $(\b_j,\b_{j+1})$ 
\end{corollary}
\begin{proof}
It follows by the continuity of the map $\a_{j,p}(\b_{j-1},\b_j)$ and the uniqueness result for ODE.
\end{proof}

\begin{theorem}
For $p$ sufficiently large there exists a radial solution $u_{p,k\text{layer}}$ of \eqref{main} having exactly $k$ maximum points $\a_{1,p},\ldots,\a_{k,p}$. Furthermore, $u_{p,k\text{layer}} \to u_{\infty,k\text{layer}}$ pointwise, as defined in Theorem \ref{thm:one_peak_solution}.
\end{theorem}
\begin{proof}
Let $T$ be as in \eqref{eq:T_def} and let $M_p=(M_p^{(1)},\ldots,M_p^{(k-1)}):T\to\R^{k-1}$, defined as
\begin{equation}\label{eq:M_p_def}
M_p^{(j)}(\b_1,\ldots,\b_{k-1})
= u_{p,1\text{-layer}}(\b_j;\b_j,\b_{j+1})-u_{p,1\text{-layer}}(\b_j;\b_{j-1},\b_{j}) 
\end{equation}
for $j=1,\ldots,k-1$.

Let us consider a domain $U\subset T$ such that $\left(M_\infty\left(\b_1,\ldots,\b_{k-1})\right)\right)^{-1}(0)\subset U$. Relation \eqref{6} and the excision property of the topological degree imply
\begin{equation}\label{8}
deg\left(M_\infty(\b_1,\ldots,\b_{k-1}), U,0\right)=1.
\end{equation}
Finally, since $M_p\rightarrow M_\infty$ uniformly in $U$, we get that
\begin{equation}\label{9}
deg\left(M_p(\b_1,\ldots,\b_{k-1}), U,0\right)=1,
\end{equation}
so that $M_p$ admits at least one zero in $U$.
\end{proof}

%


\bibliographystyle{abbrv}
\bibliography{biblio.bib}

\end{document}